\documentclass[12pt,leqno]{amsart}
\input{cyracc.def}
\usepackage[all]{xy}
\usepackage{amssymb,mathrsfs,euscript,enumitem}
\usepackage{mathrsfs,euscript,amssymb,amsmath,bbold}
\usepackage[hypertex]{hyperref}
\textheight9in  \def\DATE{February 20, 2017}
\font\cyr=wncyi8

\newtheorem{theorem}{Theorem}
\newtheorem{corollary}[theorem]{Corollary}

\newtheorem{lemma}[theorem]{Lemma}
\newtheorem{proposition}[theorem]{Proposition}

\theoremstyle{definition}
\newtheorem{example}[theorem]{Example}
\newtheorem{remark}[theorem]{Remark}
\newtheorem{definition}[theorem]{Definition}

\DeclareMathOperator{\Spec}{Spec}
\DeclareMathOperator{\ad}{ad}
\DeclareMathOperator{\Ad}{Ad}
\DeclareMathOperator{\GL}{GL}

\DeclareMathOperator{\Sh}{Sh}

\def\Ass{{\EuScript A{\it ss\/}}} 
\def\Com{{\EuScript C{\it om\/}}}
\def\Lie{{\EuScript L{\it ie\/}}}
\def\oP{{\EuScript P}}
\def\oQ{{\EuScript Q}}
\def\cS{{{}^cS}}
\def\cF{{{}^cF_{\oP}}}
\def\Fr{{F_{\oQ}}}
\def\unsh{\Sh^{-1}}

\def\ring{{T}}
\def\frn{{\mathfrak n}}
\def\gl{\mathfrak{gl}}
\def\dred{{\overline \delta'}}
\def\bardelta{{\bar\delta}}
\def\MC{{\rm QME}}
\def\frm{{\mathfrak m}}

\def\ttIBL{{\tt IBL}}
\def\L{{\tt L}}
\def\Im{{\rm Im}}
\def\un{{\sf 1}}
\def\sfe{{\sf e}}
\def\hot{{\widehat \otimes}}
\def\htt{{\widehat \times}}
\def\BV{{\rm BV}}
\def\IBL{{\rm IBL}}
\def\MV{{\rm MV}}
\def\ttMV{{\tt MV}}
\def\Ker{{\rm Ker}}
\def\bfk{{\mathbb k}}
\def\id{{\it id}}
\def\Lin{{\it Lin}}
\def\ot{\otimes}

\newcommand{\abs}[1]{{\lvert#1\rvert}}

\catcode`\@=11

\def\@evenfoot{\rule{0pt}{20pt}[\DATE] \hfill [{\tt \jobname.tex}]}
\def\@oddfoot{\rule{0pt}{20pt}{[\tt \jobname.tex}]\hfill [\DATE]}
\catcode`\@=13

\title{The MV formalism for ${\rm IBL}_\infty$- and ${\rm
    BV}_\infty$-algebras}

\author{Martin Markl and Alexander A. Voronov}

\dedicatory{{\cyr \cyracc
    \hfill    Odnazhdy chukcha prin\"es v redaktsiyu svo{\u i} roman. Redaktor proch\"el i govorit:\\
    \hfill    --- Ponimaete li, slabovato... Vam by klassiku pochitatp1.\\
    \hfill    Vy Turgeneva chitali? A Tolstogo? A Dostoevskogo?\hphantom{m}\\
    \hfill --- Odnako, net: chukcha --- ne chitatelp1, chukcha ---
    pisatelp1}.
  \smallskip\\
  \hfill \scriptsize{Russian folklore} \hphantom{mm}}

\catcode`\@=11
\address{Mathematical Institute of the Academy, {\v Z}itn{\'a} 25,
         115 67 Prague 1, The Czech Republic}
\address{Faculty of Mathematics and Physics, Charles University,
186 75 Sokolovsk\'a 83, Prague~8, The Czech Republic}
\email{markl@math.cas.cz}

\address{School of Mathematics, University of Minnesota, 206 Church
  St. S.E., Minneapolis, MN 55455, USA}
\address{Kavli IPMU (WPI), UTIAS, The University of Tokyo, Kashiwa,
  Chiba 277-8583, Japan}
\email{voronov@umn.edu}
\catcode`\@=13

\thanks{The first author was supported by the Eduard \v Cech Institute
  P201/12/G028 and RVO: 67985840. The second author was supported by
  the World Premier International Research Center Initiative (WPI
  Initiative), MEXT, Japan, and a Collaboration grant from the Simons
  Foundation (\#282349).}


\begin{document}
\bibliographystyle{plain}
\baselineskip 16.25pt  plus 1.5pt minus .5pt

\begin{abstract}
  We develop a new formalism for the Quantum Master Equation
  $\Delta e^{S/\hbar} = 0$ and the category of $\IBL_\infty$-algebras
  and simplify some homotopical algebra arising in the context of
  oriented surfaces with boundary. We introduce and study a category
  of \MV-algebras, which, on the one hand, contains such important
  categories as those of $\IBL_\infty$-algebras and
  $\L_\infty$-algebras, and on the other hand, is homotopically
  trivial, in particular allowing for a simple solution of the quantum
  master equation. We also present geometric interpretation of our
  results.
\end{abstract}


\maketitle

\tableofcontents

\section{Introduction}

Recent developments in String Topology, Symplectic Field Theory, and
Lagrangian Floer Theory have led to a new wave of homotopical algebra,
heavily burdened by formulas that seem overwhelming to the eye of an
unpretentious mathematician, see \cite{cl,cfl}. The algebra is
associated to a homotopy version of involutive Lie bialgebras, or
$\IBL_\infty$-algebras. The algebraic structure in question is
governed by surfaces with boundary. It arises from multilinear
operators that correspond to diffeomorphism classes of connected,
oriented surfaces of various genera and with various numbers of
labelled boundary components, separated into ``inputs'' and
``outputs,'' like in Topological Quantum Field Theory (TQFT). Since
such a diffeomorphism class is determined by the genus and the number
of inputs and outputs, the algebraic structure is determined by linear
operators: $l^g_{m,n}: S^m (U) \to S^n(U)$, $g \ge 0$, $m, n \ge 1$,
between symmetric powers of a graded vector space $U$. However, unlike
in TQFT, the correspondence is not supposed to be a functor:
$l^0_{1,1}: U \to U$ has to be a differential on $U$ rather than the
identity, and $l^g_{m,n}$ must be a null-homotopy for the sum of all
possible gluings of two surfaces along part of their inputs and
outputs resulting in a surface of genus $g$ with $m$ inputs and $n$
outputs, see \cite{dctt,terilla:quantizing,cfl}. These operations,
$l^g_{m,n}$ may be collected nicely into a generating function
$\Delta: S(U)[[\hbar]] \to S(U)[[\hbar]]$, called a BV (or
$\BV_\infty$) operator, which is in fact an odd differential operator
such that $\Delta^2 = 0$, see \cite{dctt,cfl,munster-sachs}. From the
operadic perspective, we are talking about the differential graded
(dg) dual notion to the notion of a TQFT, namely an algebra over the
dg dual properad to the Frobenius algebra properad, which is Koszul
dual to the involutive Lie bialgebra properad, see
\cite{dctt,cmw}. The operadic yoga applies and shows that the
algebraic structure is a canonical homotopy version of the structure
of an involutive Lie bialgebra, producing an (almost) canonical name:
that of an $\IBL_\infty$-algebra.

In concrete geometric contexts, the construction of an $\IBL_\infty$
structure involves a number of choices of geometric data, such as an
almost complex structure on the target space, a Riemannian metric on
the surface, \&c., and showing the independence of the algebraic
structure on the choices invokes the notion of an
$\IBL_\infty$-morphism. That notion is quite elaborate algebraically,
see \cite{cfl}, but may be packed nicely in a generating function,
involving the notion of an ``exponential'' of a linear map
$S(U') \to S(U'')$ between symmetric algebras. The exponential was
defined in \cite{cl,cfl} by an explicit formula, see \eqref{exp},
bearing a certain resemblance to the exponential power series, as well
as respecting the exponentials of elements in these symmetric
algebras, see \cite{sasha's}, Example~\ref{za_14_dni_Izrael} and
Theorem~\ref{BV_transfer}.

The current work came out of the authors' realization that the
exponential of a map was a true exponential in a (graded) commutative
algebra given by the convolution product, see Equations
\eqref{Psano_dole_v_hotelu} and \eqref{exp}. This led to a significant
simplification of dealings with the category of $\IBL_\infty$-algebras
as well as to a ``larger'' category of \MV-algebras, which we present
in this paper. We show that the category of \MV-algebras is equivalent
to a certain category of pointed complexes. Thus, the category of
\MV-algebras is rather trivial as a category, but this does not
prevent it from having room for such highly nontrivial subcategories
as those of $\L_\infty$- and $\IBL_\infty$-algebras. This categorical
triviality may be regarded as homotopy triviality, presenting itself
through the observation that the quantum master equation \eqref{QME}
in an \MV-algebra is just a cocycle condition. We also show that
$\IBL_\infty$-morphisms are closed under composition within the
category of \MV-algebras, which seems to be a nontrivial property,~cf.~\cite{cfl}.

The theme of the current article may also be viewed as the study of
the quantum master equation (QME)
\begin{equation}
\label{QME}
\Delta e^{S/\hbar} = 0,
\end{equation}
which originates in the BV formalism, see e.g.~\cite{losev}, and takes
place in a dg BV- or, more generally, (hereafter, commutative)
$\BV_\infty$-algebra. $\IBL_\infty$-algebras provide an important, but
not exhaustive class of examples of $\BV_\infty$-algebras: those for
which the underlying commutative algebra is free, see Examples
\ref{sec:mv-algebras-1} and \ref{Treti_den_pekla}. The QME may be
viewed as an $\hbar$-deformation of the Maurer-Cartan equation
$dS + \frac{1}{2} [S,S] = 0$, cf.\
Theorem~\ref{sec:quant-mast-equat}. Thus, the study of solutions of
the QME must be related to a quantization of deformation theory, cf.\
\cite{terilla:quantizing}, in which the question of homotopy
invariance of solutions with respect to weak equivalences would be one
of the primary problems. The homotopy category of dg BV- and
$\BV_\infty$-algebras should be one of the working notions of
quantized deformation theory. Cieliebak and Latschev \cite{cl} defined
the notion of a $\BV_\infty$-morphism, only when the source was free
as a commutative algebra, i.e.\ derived from an
$\IBL_\infty$-algebra. Such a mismatch of sources and targets prevents
$\BV_\infty$-algebras from forming a category. \MV-morphisms described
in this paper include such $\BV_\infty$-morphisms and, on the other
hand, form a category. Moreover, the QME makes perfect sense in the
more general context of \MV-algebras, see
Section~\ref{sec:master-equation-1}, which makes them arguably a
better candidate for providing a background for quantum deformation
theory. The main results of the paper in this direction are
Theorem~\ref{sec:master-equation-2}, which may be interpreted as a
representability theorem for the functor of solutions of the QME, see
Remark~\ref{representability}, and a description of the transfer of
solutions of QME, Theorem~\ref{BV_transfer}. We also present geometric
interpretation of our results.

\noindent 
{\bf Disclaimer.}
It may appear that our choice of terminology goes against the good old
mathematical tradition of being modestly egocentric. We should assure
the reader that, on the contrary, we have chosen to be modest about
the mathematical content of the paper and, at the same time, allowed
ourselves to be somewhat egocentric in such a minor, cosmetic issue as
terminology.

\noindent 
{\bf Conventions.}
The symbol $\bfk$ will denote a fixed field $\bfk$ of characteristic
zero and $\ot$ the tensor product over $\bfk$. We will denote by
$\id_X$ or simply by $\id$ when $X$ is understood, the identity
endomorphism of an object $X$ (vector space, algebra, \&c.).  We
will sometimes denote the product of elements $a$ and $b$ of an
algebra using the explicit name of the multiplication (typically
$\mu(a,b)$), sometimes, when the meaning is clear from the context, by
$a\cdot b$, or simply by $ab$.

\def\redukce#1{\vbox to .3em{\vss\hbox{#1}}} For a graded
$\bfk$-vector vector space $V$ and a complete local commutative ring
$R$ with the residue field $\bfk$ and the maximal ideal $\frm$ we
denote by $V \hot R$ the completed tensor product $\varprojlim_n V\!
\otimes\! R/V\! \ot\!\frm^n$. In the particular case when $R$ is the
formal power series ring $\bfk[[\hbar]]$ in $\hbar$ we abbreviate, as
usual, \redukce{$V \hot \,\bfk[[\hbar]]$} by $V[[\hbar]]$. We also
abbreviate $\bfk(\!(\hbar)\!) : = \bfk[[\hbar]][\hbar^{-1}]$ and
$V(\!(\hbar)\!) := V \hot\, \bfk(\!(\hbar)\!)$.  The degree of a
homogeneous element $v \in V$ is denoted by $\abs{v}$.

\def\nni{non-negative integer}
\def\rada#1#2#3{{#1}_{#2},\ldots,{#1}_{#3}}
Recall that, for graded indeterminates  $\rada u1n$ and a permutation $\sigma\in
\Sigma_n$, the {\em Koszul sign\/} $\epsilon(\sigma)
=\epsilon(\sigma;\rada u1n) \in \{+1,-1\}$ is defined by
\[
u_1\odot\dots\odot u_n = \epsilon(\sigma;u_1,\dots,u_n)
\cdot u_{\sigma(1)}\odot \dots \odot u_{\sigma(n)}
\]
which has to be satisfied in the free graded commutative algebra
$S(\rada u1n)$ generated by $\rada u1n$.
If $k\geq 1$ and $a_1,\ldots,a_k$ are \nni{}s such that
$a_1+\cdots + a_k = n$, then an
$(a_1,\ldots,a_k)$-{\em unshuffle\/} of $n$
elements is a permutation $\sigma \in \Sigma_n$ such that
\[
\sigma(1) < \cdots <\sigma(a_1),\ \sigma(a_1+1) < \cdots< \sigma(a_1+a_2),
\cdots,
\ \sigma(n-a_k+1) < \cdots <\sigma(n).
\]
The subset of all $(a_1,\ldots,a_k)$-unshuffles will be denoted by
$\unsh_{\rada a1k} \subset \Sigma_n$.

\vskip .2em
\noindent {\bf Acknowledgment.}
The authors are indebted to D.~Bashkirov and the anonymous referees for
useful remarks, and to J.~Latschev for sharing a preliminary version
of his paper with K.~Cieliebak and K.~Fukaya. The first author also
wishes to express his thanks to M.~Doubek, B.~Jur\v co and I.~Sachs
for introducing him to the jungle of $\IBL_\infty$-algebras, and
acknowledge the hospitality of the University of Minnesota where he
held the position of a Distinguished Ordway Visitor during the last
stages of the work on this article. The second author gratefully
acknowledges support from the Graduate School of Mathematical
Sciences, the University of Tokyo, and the Simons Center for Geometry
and Physics, Stony Brook University, at which some of the research for
this paper was performed.

\section{MV-algebras}

\renewcommand*{\thefootnote}{\fnsymbol{footnote}}
\setcounter{footnote}{0}

\begin{definition}
\label{Prvni_den_po_navratu_z_Minneapolis}
Let $R$ be a complete local Noetherian commutative ring with the
residue field $\bfk$ and a maximal ideal $\frm$.  An
{\em\MV-algebra\/}\footnote{Abbreviating Markl-Voronov. Not to be
  mistaken with MV-algebras occurring in many-valued logic.}  over $R$
is a quadruple $V= (V,\mu,\delta,\Delta)$ consisting~of
\begin{itemize}
\item[(i)] a unital graded associative commutative $\bfk$-algebra
  $(V,\mu)$,
\item[(ii)] a \emph{conilpotent} counital graded coassociative
  cocommutative $\bfk$-coalgebra $(V,\delta)$, and
\item[(iii)] a continuous degree $+1$ $R$-linear map $\Delta : V\hot R
  \to V\hot R$ such that $\Delta^2=0$ and $\Delta(1) = 0$.
\end{itemize}
We require that the unit map $\eta : \bfk \to V$ for $(V,\mu)$ be a
coaugmentation for $(V,\delta)$ and the counit $\epsilon : V \to \bfk$
be an augmentation for $(V,\mu)$. These conditions imply that these
maps are nontrivial and $V \ne 0$.  In most situations, $R$ will
either be the ground field $\bfk$ or a power series ring
$\bfk[[\hbar]]$.
\end{definition}

\begin{remark}
Let $I = \Ker (\epsilon)$ be the augmentation ideal of $(V,\mu)$. The
{\em reduced diagonal\/} $\bardelta$ is defined by the standard formula
\[
\bardelta(v) := \delta(v) - (v \ot 1) - (1\ot v)
\]
for $v\in I$, while  $\bardelta(1) := 0$. We iterate
$\bardelta$ by putting $\bardelta^{[0]}(v) := v$ and
\[
\bardelta^{[k]}(v) : = \big(\bardelta \ot \id^{\ot k-1} \big)\bardelta^{[k-1]}(v)
\]
for $k \geq 1$.  The {\em conilpotency\/} of $(V,\delta)$ means that
for each $v \in V$, $\bardelta^{[k]}(v) = 0$ for $k$ large enough. Notice
that we {\em do not\/} require any compatibility between $\mu$ and
$\delta$, which would be required of $(V,\mu,\delta)$ to be a
bialgebra, a Frobenius algebra, \&c.
\end{remark}

\begin{remark}
\label{sec:mv-algebras-3}
Each $\MV$-algebra over $\bfk$ can be considered as an $\MV$-algebra
over $R$ after the $R$-linear extension of $\Delta$. In this manner we
obtain precisely those $\MV$-algebras over $R$ for which $\Delta(V)
\subset V$.
\end{remark}

\begin{example}
\label{sec:mv-algebras-2}
Every augmented unital commutative associative $\bfk$-algebra $A$ is
an \MV-algebra over $\bfk$ with the comultiplication defined by
$\delta(1) := 1 \ot 1$, while $\delta(v) := v \ot 1 + 1 \ot v$ for $v$
in the augmentation ideal of $A$, and $\Delta:= 0$.
\end{example}

\begin{example}
\label{sec:mv-algebras}
Dually, a conilpotent cocommutative coassociative $\bfk$-coalgebra
$(V,\delta)$ with a counit $\epsilon : V \to \bfk$ and a coaugmentation
$\eta: \bfk \to V$ is an \MV-algebra over $\bfk$ with a unit $\eta$,
augmentation $\epsilon$ and multiplication given by $1 \cdot v = v
\cdot 1:= v$ for each $v \in V$, and $v'\cdot v'' = 0$ for $v',v'' \in
\Ker(\epsilon)$. We again put $\Delta:= 0$.
\end{example}

\begin{example}
\label{supertrivial}
Each graded $\bfk$-vector space $V$ equipped with linear maps
$\epsilon : V \to \bfk$ and $\eta : \bfk \to V$ such that $\epsilon
\circ \eta = \id$ bears the `supertrivial'  unital algebra and
counital coalgebra
structures defined similarly as in the above two examples. Such $V$ is
therefore a `supertrivial' MV-algebra with any $R$-linear 
$\Delta: V \hot R \to V\hot R$ satisfying $\Delta^2 = 0$ and $\Delta
\circ \eta = 0$.  
\end{example}

\begin{example}
\label{Jaruska_si_maluje}
Let $R$ be such that the quotient $R/\frm^p$ is a finite-dimensional
$\bfk$-vector space for each $p \geq 1$.  The continuous $\bfk$-linear
dual $R^*$, with the comultiplication given by the dual of the
multiplication in $R$, is conilpotent by the completeness of $R$. In
particular, $R^*$ itself with $\Delta : = 0$ and trivial
multiplication, as in Example~\ref{sec:mv-algebras}, is an \MV-algebra
over $\bfk$.
\end{example}

\begin{example}
\label{uz_ctyri_hodiny}
A less trivial example of an \MV-algebra over $\bfk$ is obtained by
taking $V := S(U)$, the (graded) symmetric, or polynomial, algebra
generated by a graded vector space $U$, with the standard bialgebra
structure, and $\Delta$ a degree $1$ coderivation with $\Delta^2 = 0$
and $\Delta(1) = 0$. Such a structure is a disguise of a {\em strongly
  homotopy Lie\/} (${\rm L}_\infty$-) algebra
\cite[Theorem~2.3]{lada-markl:CommAlg95}.
\end{example}

Let $A$ be a unital associative commutative algebra and $\Delta :A \to
A$ be a $k$-linear map. For $n \geq 0$, consider the iterated graded
commutators
\begin{equation}
  [[\ \dots\ [\Delta, L_{a_1}], 
  \hbox {$. \hskip .05em  . \hskip .05em.$}], L_{a_n} ],
\end{equation}
$L_a$ denoting the operator of left multiplication by $a \in A$. By
convention, we just set the commutator of $\Delta$ with $n = 0$
left-multiplication operators to be $\Delta$.
We call an operator $\Delta$ an {\em order $\le k$ differential
  operator\/} if the iterated commutator with any $k+1$
left-multiplication operators vanishes.

Now suppose that $\Delta(1) = 0$. Define
\begin{equation}
\label{psano_v_Haife}
\Phi_n^\Delta(a_1,\ldots,a_n) := [[\
\dots\ [\Delta, L_{a_1}], 
\hbox {$. \hskip .05em  . \hskip .05em.$}], L_{a_n} ]( 1).
\end{equation}
In particular, $\Phi^\Delta_0 = 0$ and $\Phi^\Delta_1 = \Delta$. If
$\Phi_n^\Delta = 0$ for $n > k$, the operator $\Delta$ is called an
{\em order $k$ derivation\/}~\cite[Section~1.2]{markl:ab}). Notice
that first-order derivations are vector fields (derivations).

\begin{example}
\label{sec:mv-algebras-1}
Recall~\cite[Definition~7]{kravchenko1} that a {\em $($commutative$)$
  $\BV_\infty$-algebra\/} consists of a unital graded commutative
associative algebra $A$ and a $\bfk[[\hbar]]$-linear degree $1$ map
$\Delta : A[[\hbar]] \to A[[\hbar]]$ such that $\Delta^2 = 0$ and
$\Delta(1)=0$.\footnote{The continuity of $\Delta$ is automatic
  because $\bfk[[\hbar]]$ is regular, see~\cite[Proposition
  1.11]{markl12:_defor}.}  One moreover requires that $\Delta$
decomposes into a sum
\begin{equation}
\label{sec:mv-algebras-1bis}
\Delta = \Delta_1 + \hbar\Delta_2 + \hbar^2\Delta_3 + \cdots,
\end{equation}
where $\Delta_k : A \to A$ is a $\bfk$-linear order $\le k$
differential operator on $A$. If we assume that $A$ has an
augmentation as a graded commutative algebra and equip $A$ with the
comultiplication $\delta$ constructed in
Example~\ref{sec:mv-algebras-2}, then $A= (A,\mu,\delta,\Delta)$
becomes an \MV-algebra over~$\bfk[[\hbar]]$. The most common degree
convention is $\abs{\hbar} = 2$, which is implicit in
\cite{kravchenko1} and explicit in \cite{sasha's}, but one can
consider $\hbar$ of degree zero, as in \cite{munster-sachs}, or
arbitrary even degree, as in \cite{cl,cfl}. To comply with our
convention that the ground ring $R$ is not graded, we set $\abs{\hbar}
= 0$.
\end{example}

\begin{example}
\label{Treti_den_pekla}
Let $A :=S(U)$ be the symmetric algebra generated by a graded vector
space~$U$. An {\em $\IBL_\infty$-algebra\/} structure on
$U$~\cite{cfl}, 
\cite[\S4.2]{munster-sachs} is given by a degree $1$, $\bfk[[\hbar]]$-linear
map $\Delta : A[[\hbar]] \to A[[\hbar]]$ satisfying $\Delta^2=0$, $\Delta(1) =
0$, which decomposes as in~(\ref{sec:mv-algebras-1bis}).  In other
words, an $\IBL_\infty$-algebra is a $\BV_\infty$-algebra whose
underlying augmented commutative algebra is $S(U)$.  If we equip
$S(U)$ with the standard bialgebra structure,
$\big(S(U),\mu,\delta,\Delta\big)$ will be another example of an \MV-algebra
over $\bfk[[\hbar]]$.
\end{example}

\begin{remark}[Geometric interpretation, see Table~\ref{analogies}]
  One can think of an \MV-algebra $V$ over $R$ as the algebra of
  functions on a family, that is to say a fiber bundle, $X \to B$ of
  graded manifolds $B = \Spec(R)$, $X = \Spec(V\hot R)$, 
endowed with a square-zero differential operator
  $\Delta$ of degree one on $X$ over $B$, thereby, generalizing the
  notion of a differential graded manifold, when the operator happens
  to be a derivation linear over functions on $B$. This fiber bundle
  has a section $B \to X$, thought of as a family of basepoints, and
  the space of distributions on $X$ over $B$ is provided with a graded
  commutative product, not necessarily related to multiplication of
  functions on $X$, cf.~Table~\ref{analogies} for the origin of these
  geometric data. 
The only relationship between the products on
  functions and distributions we place is that the projection $X \to
  B$ and the section $B \to X$ must provide an augmentation and
  a~unit, respectively, for multiplication of distributions. This
  geometric object may be called an \emph{\MV-manifold\/}. A typical
  situation is when the fiber bundle is trivial: $F \times B \to B$,
  and one can think of an \MV-manifold as a $B$-parameterized family
  of differential operators on~$F$. See more on geometric analogies in
  Table~\ref{analogies}.
\begin{table}
\def\arraystretch{1.3}
\begin{center}
  \begin{tabular}{|l|l|} \hline \multicolumn{1}{|c|}{\textbf{Algebra}} & \multicolumn{1}{c|}{\textbf{Geometry}}\\
    \hline Ring $R$ & Base manifold $B$ \\
    \hline Graded algebra $V \hot R$ & Graded manifold $X$ \\
    \hline Unit $\eta: R \to V \hot R$ & Fiber bundle $X \to B$\\
    \hline Operator $\Delta$ & Differential operator on $X$ over $B$\\
    \hline Augmentation $\epsilon: V \hot R \to R$ & Section $B \to X$\\
    \hline The $R$-linear dual $\Lin_R(V \hot R, R)$ &  Distributions on $X$ over $B$\\
    \hline Comultiplication on $V \hot R$ & Multiplication of distributions on $X$ \\
    \hline
\end{tabular}
\end{center}
\caption{\protect\rule{0em}{1.8em}Dictionary of Geometric Analogies}
\label{analogies}
\end{table}

Since the definition is essentially symmetric with respect to the
algebra-coalgebra structures, we can dually consider the graded
algebra structure on the $\bfk$-linear dual $V^*$ defined by the
graded coalgebra structure on $V$ and think of $V^*$ as the algebra of
functions on a family $X^* \to B$ of ``formal pointed graded
manifolds'' endowed with a square-zero, odd ``differential operator''
over $B$ and the structure of a graded commutative algebra on the
space of distributions on $X^*$, along with similar compatibility
conditions between the units and counits.  We will call such geometric
objects \emph{dual \MV-manifolds\/}.  Depending on the situation,
either interpretation could be preferable. For instance, in
Example~\ref{uz_ctyri_hodiny}, the \MV-algebra $V=S(U)$ corresponding
to an ${\rm L}_\infty$-algebra $U[-1]$, where $U[-1]$ denotes an
appropriate degree shift, is usually interpreted as a \emph{formal
  differential graded manifold $U$\/}. This is an example of a dual
\MV-manifold $X = \Spec V^*$. On the other hand, in
Example~\ref{sec:mv-algebras-1}, the \MV-algebra $A$ may rather be
interpreted as a $\BV_\infty$-\emph{manifold\/} $\Spec A$, i.e.\ a
family of graded manifolds with an odd, square-zero differential
operator over $\Spec \bfk[[\hbar]]$. This is an example of an
\MV-manifold $\Spec A$.

Geometric objects of this nature may arise in various situations,
starting from a graded manifold or graded scheme $X$ provided with
extra structure, such as those in the following examples:
\begin{itemize}[leftmargin=0.8em,labelsep=0.4em,itemsep=.3em]
\item $X$ is a graded ``abelian Lie group,'' see Examples
  \ref{uz_ctyri_hodiny} and \ref{Treti_den_pekla}, in which $X$ is
  moreover a vector space (or a vector bundle over $B$).
\item $X$ has a ``volume density'' inducing a graded Frobenius algebra
  structure on the space of functions on $X$.
\item $X$ has a basepoint $x_0$ and is provided with a ``trivial''
  multiplication law on distributions so that distributions vanishing
  on constants multiply to zero and the delta function $\delta_{x_0}$
  serves as a unit, see Examples \ref{sec:mv-algebras-2} and
  \ref{sec:mv-algebras-1}. In particular, the delta functions of
  points multiply as follows: $\delta_x \cdot \delta_y := \delta_x +
  \delta_y - \delta_{x_0}$ for $x, y \in X$. In this case $X^*$ is a
  point with a ring $V^*$ of functions such that the maximal ideal
  $I^*$ of this point has trivial multiplication, $(I^*)^2 = 0$.
\item Dually, $X$ may have the ``infinitesimal'' geometry of a point
  enriched with a ring $V$ of functions such that the maximal ideal
  $I$ has trivial multiplication $I^2 = 0$, whereas the space $V^*$ of
  distributions may have an interesting multiplication, reflected in
  some nontrivial geometry of $X^*$, as in
  Example~\ref{sec:mv-algebras}.
\end{itemize}
\end{remark}

\section{The arithmetic of convolution product}
\label{sec:morphisms}

Let $V'= (V',\mu',\delta',\Delta')$,
$V''= (V'',\mu'',\delta'',\Delta'')$ be two MV-algebras, and let
$\Lin_\bfk(V',V'')$ be the set of degree-$0$ $\bfk$-linear maps
between their underlying spaces. It is well-known, see e.g.~\cite[\S
III.3]{kassel}, that the comultiplication $\delta'$ together with the
multiplication $\mu''$ induces on $\Lin_\bfk(V',V'')$ the structure of
a unital commutative associative augmented algebra, via the {\em
  convolution product\/}~$\star$.  Explicitly
\begin{equation}
\label{Psano_dole_v_hotelu}
f \star g := \mu'' (f \ot g) \delta'
\end{equation}
for $f,g \in \Lin_\bfk(V',V'')$. The unit $\sfe = \sfe_{V',V''}$ for
$\star$ is the composition $\eta''
\circ \epsilon'$ of the augmentation of $V'$ with the unit of
$V''$. The algebra  $\Lin_\bfk(V',V'')$ is augmented by the map that 
sends $f \in  \Lin_\bfk(V',V'')$ to
$\epsilon''\big(f(1)\big) \in \bfk$. Notice that
\begin{equation}
\label{eq:2}
\sfe_{V'',V'''} \circ \sfe_{V',V''} = \sfe_{V',V'''},\
\ \sfe_{V',V''} \circ \eta' = \eta'',\
\epsilon'' \circ \sfe_{V',V''} = \epsilon' .
\end{equation}

The above constructions clearly extend by $R$-linearity to the space
$\Lin_R(V' \hot R,V'' \hot R)$ of continuous $R$-linear maps which we
will, for brevity, denote by $\Lin_R(V',V'')$ believing that the
reader will not be too confused by this shorthand.

Let $\frm$ be the maximal ideal of $R$.  Denote by $\Lin_R^0(V',V'')$
the subset of $\Lin_R(V',V'')$ consisting of continuous $R$-linear
maps $f : V' \hot\, R \to V''\hot\, R$ such that
\begin{equation}
  \label{eq:1}
  \Im(f \circ \eta') \subset V'' \hot \, \frm.\footnote{We used the
    same symbol for the map $\eta'$ and its $R$-linear extensions. We keep this
    convention throughout the paper.}
\end{equation}
Notice that for $f,g
\in  \Lin_R^0(V',V'')$, 
\begin{equation}
\label{Psano_dole_v_hotelu_1}
f \star g \equiv \mu'' (f \ot g) \bardelta' \ \hbox{ mod } \ V'' \hot\, \frm, 
\end{equation}
where $\bardelta'$ is the reduced diagonal associated to $\delta'$,
and moreover $\Lin_R^0(V',V'')$ is an ideal in $\Lin_R(V',V'')$. We
leave to prove as an exercise that the
conilpotency of $\delta'$ together with the completeness of $R$
implies:

\begin{lemma}
  All power series in elements of $\Lin_R^0(V',V'')$
  converge.\footnote{Convergence is always understood as convergence
    in the $\frm$-adic topology.}
\end{lemma}

In particular, for  $f \in \Lin_R^0(V',V'')$ it makes sense to take
\begin{equation}
\label{Zitra_budeme_s_Jaruskou_pit_portske.}
\exp(f) := \sfe + f + \frac {f^2}{2!} + \frac {f^3}{3!} + \cdots \in
\Lin_R(V',V'') \quad \hbox {(powers w.r.\ to the $\star$-product).}
\end{equation}
Notice that, while $\exp(f) \notin \Lin_R^0(V',V'')$, clearly
$\exp(f)-\sfe \in \Lin_R^0(V',V'')$.

If $g \in \Lin_R^0(V',V'')$, then $\sfe + g \in \Lin_R(V',V'') $
and, thanks to~\eqref{eq:2}, we have
\[
(\sfe + g) \circ \eta' \equiv \eta'' \ \hbox { mod } \ V''\hot\,\frm
\]
and the power series
\[
\log(\sfe + g) := g - \frac{g^2}{2} + \frac{g^3}{3} - \dots
\in \Lin_R^0(V',V'')
\]
converges.  For $f, g \in \Lin_R^0(V',V'')$, one clearly has the
expected equalities
\[
\log (\exp(f)) = f, \quad \exp(\log(\sfe + g)) = \sfe + g.
\]

\begin{example}
\label{za_14_dni_Izrael}
Let $G(V') := \{v \in V' \; | \; \delta'(v) = v \ot v \}$ be the
subset of group-like elements in $V'$, and $v \in G(V') \hot R$. Then,
for any $f,g \in \Lin_R(V',V'')$,
\[
(f \star g)(v) =
\mu'' (f \ot_R g) \delta'(v) =
\mu'' (f \ot_R g)(v \ot_R v) = \mu''\big(f(v),g(v)\big), 
\]
so $(f \star g)(v)$ is the `actual' product of the elements $f(v)$ and
$g(v)$ in the algebra $V''$. Consequently, for $f \in
\Lin^0_R(V',V'')$ and $v \in G(V') \hot R$,
\[
\exp(f)(v) = e^{f(v)},
\]
the `actual' exponential of $f(v)$ in the algebra $V''$; similarly for
the logarithm.
\end{example}

\begin{proposition}
\label{multiplicativity}
Given two MV-algebras $V'$ and $V''$ and $f \in \Lin_R^0(V',V'')$,
suppose that $V'$ is in fact a bialgebra, i.e.\ $\delta'$ is an
algebra morphism. Then, for any $v_1, v_2 \in V'$,
\begin{equation}
\label{v_Myluzach}
\exp(f) (v_1 v_2) \equiv \exp (f)(v_1) \exp (f) (v_2) \mod \big(V'' \hot
\frm,f(I'^2)\big),
\end{equation}
where $I'$ is the augmentation ideal of $V'$. In other words,
$\exp(f)$ is an algebra morphism modulo the ideal in $V'' \hot R$
generated by $V'' \hot \frm$ and $f(I'^2)$.
\end{proposition}

\begin{proof}
Since~\eqref{v_Myluzach} obviously holds if $v_1$ or $v_2$ are
proportional to $1\in V'$, we will assume that $v_1,v_2 \in I'$. In
this case, since $\epsilon'(v_1) =  \epsilon'(v_2) = 0$, 
$\exp(f)(v_1) \exp (f)(v_2) $ is the product
\begin{equation}
\label{Vcera_vylet}
 \big(f(v_1) + \frac{1}{2!} f^2(v_1) + \frac{1}{3!} f^3(v_1) +
 \cdots\big)\big(f(v_2) + \frac{1}{2!} f^2(v_2) + \frac{1}{3!} f^3(v_2) + \cdots\big)
\end{equation}
while
\begin{equation}
\label{Zas_jsem_podlehl}
 \exp (f)(v_1v_2) = f(v_1v_2) + \frac{1}{2!} f^2(v_1v_2) + \frac{1}{3!} f^3(v_1v_2) +
 \cdots.
\end{equation}
As $\delta'$ is, by assumption, an algebra morphism, one has
\begin{align*}
  \delta'(v_1v_2)& = \delta'(v_1)\delta'(v_2) = \big(v_1\ot 1 +
  \dred(v_1) + 1\ot v_1\big)\big(v_2\ot 1 + \dred(v_2) + 1\ot v_2\big)
  \\
  &= v_1\ot v_2 + (-1)^{\abs{v_1} \cdot \abs{v_2}} v_2\ot v_1 +
  v_1v_2\ot 1 + 1 \ot v_1v_2 + \dred(v_1)\dred(v_2)
  \\
  & \quad +(1 \ot v_1)\dred(v_2) + (v_1\ot 1)\dred(v_2) +
  \dred(v_1)(1\ot v_2) + \dred(v_1)(v_2\ot 1) .
\end{align*}
Notice that $\mu'' \circ (f\ot f)$ applied to the terms on the
right-hand side of the above equation vanishes modulo $\big(V'' \hot
\frm,f(I'^2)\big)$ everywhere except at the terms $v_1\ot v_2$ and
$v_2\ot v_1$, thus
\[
f^2(v_1v_2) = \mu'' (f\ot f) \delta'(v_1v_2) \equiv \mu'' \big(f(v_1)
\ot f(v_2) + (-1)^{\abs{v_1} \cdot \abs{v_2}} f(v_2) \ot f(v_1)\big)
=2f(v_1)f(v_2)
\]
modulo $\big(V'' \hot \frm,f(I'^2)\big)$.
Similarly we obtain that
\[ 
f^3(v_1v_2) = \mu'' (f^2\ot f) \delta'(v_1v_2) \equiv 3f^2(v_1)f(v_2)
+ 3f(v_1)f^2(v_2) \mod \big(V'' \hot \frm,f(I'^2)\big)
\]
and, inductively, 
\[
f^n(v_1v_2) \equiv \sum_{1 \leq i \leq n-1} \binom{n}{i} f^i
(v_1)f^{n-i}(v_2)
\mod  \big(V'' \hot
\frm,f(I'^2)\big).
\]
This formula makes the verification that the product~(\ref{Vcera_vylet})
equals~(\ref{Zas_jsem_podlehl}) modulo $\big(V'' \hot
\frm,f(I'^2)\big)$ obvious.
\end{proof}

\begin{corollary}
\label{modm}
Under the assumptions of Proposition~\ref{multiplicativity}, suppose
also that $f(I'^2) \subset V'' \hot \frm$. Then, for any $v_1, v_2 \in
V'$,
\begin{equation}
\exp(f) (v_1 v_2) \equiv \exp (f)(v_1) \exp (f) (v_2) \mod V'' \hot \frm.
\end{equation}
Thus, $\exp(f)$ is an algebra morphism modulo $V'' \hot \frm$.
\end{corollary}

\begin{example}
Let us show that Proposition~\ref{multiplicativity} cannot be
strengthened. Take $V' = V'' = \bfk$ with the obvious bialgebra
structure. Then $f \in \Lin_R^0(\bfk,\bfk)$ is determined by $\alpha
:= f(1)
\in \frm$. By definition, $\exp (f)(v) =
v e^\alpha$ for $v \in \bfk$, so we have  
\[
\exp(f) (v_1 v_2) = v_1 v_2  e^\alpha
\]
while 
\[
\exp(f)(v_1) \exp (f) (v_2) =  v_1 v_2  e^{2\alpha}.
\]
The map $ \exp (f): \bfk \to \bfk$ is in this case an algebra morphism
only modulo the ideal generated by $\alpha \in \frm$.

As the second example, take $V' = V'' = S(U)$ with the standard
bialgebra structure, $R := \bfk$, and assume that $f :S(U) \to
S(U)$ is such that $f|_{S^n(U)} \not =0$ only for $n=2$. One then
easily calculates that, for
$v_1,v_2 \in U$,  
\[
 \exp (f)(v_1) = \exp (f)(v_2) = 0
\]
while 
\[
 \exp (f)(v_1v_2) =
f(v_1v_2).
\]
The map $ \exp (f) : S(U) \to S(U)$ is  an algebra morphism only modulo
$f\big(S^{\geq 2}(U)\big)$.

Assume now that $f|_{S^n(U)} \not =0$ only for $n=1$. Then
$f\big(S^{\geq 2}(U)\big) = 0$ and, since $R = \bfk$, also $\frm =
0$. Formula~(\ref{v_Myluzach}) therefore asserts that $ \exp (f)$ is
an algebra morphism. We leave as an exercise to verify that indeed $\exp(f)$
provides the unique extension of a linear map $U \to S(U)$ into an
algebra morphism.
\end{example}

\begin{remark}
\label{psano_v_St_Paule}
  Proposition~\ref{multiplicativity} and Corollary~\ref{modm} provide
  sufficient conditions for $\exp(f)$ to be a perturbation of an
  algebra morphism. See also Examples \ref{ibl} and \ref{bv}, in which
  $\exp(f)$ will be an $\hbar$-perturbation of an algebra morphism. In
  the general case, one can view $\exp(f)$ as a generalized algebra
  morphism.
\end{remark}

\begin{example}
\label{sec:morphisms-2}
Let us describe the exponential of an $R$-linear map $f \in
\Lin^0_R\big(S(U), B)$, where $S(U) = \big(S(U),\delta, \eta\big)$ is
the symmetric algebra generated by a~graded vector space~$U$ with the
standard cocommutative coassociative counital comultiplication, and $B
= (B,\mu,1)$ a commutative associative unital algebra.

\def\squeezedcdots{\hskip -.15em \cdot \hskip -.15em \cdot \hskip
  -.15em \cdot\hskip -.15em } Denoting, as usual, by $\delta^{[k-1]} :
S(U) \to S(U)^{\ot k}$ the diagonal iterated $(k-1)$-times and,
likewise, by $\mu^{[k-1]} : B^{\ot k} \to B$ the $R$-linear extension
of the iterated multiplication in $B$, we have, for $k \geq 1$ and $f$
as~above
\begin{equation}
\label{prsi}
f^k = \mu^{[k-1]} \circ f^{\ot k} \circ \delta^{[k-1]}.
\end{equation}
It is easy to verify that the $(k-1)$-times iterated diagonal on
the product  $u_1 \odot \cdots \odot u_n \in S^n(U)$ of
$u_1,\ldots,u_n \in U$  equals
\begin{equation}
\label{vedro}
\delta^{[k-1]}(\hskip -.05em u_1 \odot \squeezedcdots \odot u_n \hskip -.05em ) =
\sum
\frac {\epsilon(\sigma)}{a_1! \cdots a_k!}
 [u_{\sigma(1)} \odot \squeezedcdots \odot u_{\sigma(a_1)}] \ot \cdots \ot
[u_{\sigma(n-a_k+1)} \odot \squeezedcdots \odot u_{\sigma(n)}],
\end{equation}
where the summation runs over all permutations $\sigma\in\Sigma_k$ and
integers $a_1,\ldots,a_k \geq 0$
such that $a_1 +
\cdots + a_k = n$, 
and $\epsilon(\sigma)$ is the Koszul sign. 
Here we also use the convention that if $a_i = 0$, then
\[
[u_{\sigma(a_1 + \dots + a_{i-1})} \odot \cdots \odot u_{\sigma(a_1 +
  \dots + a_{i})} ] = 1.
\]
Evaluating $\delta^{[k-1]}$ in~(\ref{prsi}) using~(\ref{vedro}) gives,
for  $u_1 \odot \cdots \odot u_n \in S^n(U)$ and $n \geq 1$,
\begin{align}
\label{exp}
\exp(f)(\hskip -.05em u_1 \odot \squeezedcdots \odot u_n \hskip -.05em )& =
\\ 
\nonumber 
\sum_{k=1}^\infty \frac1{k!} &  \sum
\frac {\epsilon(\sigma)}{a_1! \cdots a_k!}
f(u_{\sigma(1)}\odot \squeezedcdots \odot u_{\sigma(a_1)})\cdots 
f(u_{\sigma(n-a_k+1)}\odot \squeezedcdots \odot u_{\sigma(n)}).
\end{align}
We recognize a formula in~\cite[Section~5]{cl}. Notice that, thanks to the
commutativity of the multiplication $\odot$ in $S(U)$, this formula can
be rewritten as
\begin{equation}
\label{Zitra_vylet_na_kole_podel_Missisipi}
\exp(f)(\hskip -.05em u_1 \odot \squeezedcdots \odot u_n \hskip -.05em
) =
\sum_{k=1}^\infty \frac1{k!} \sum
 {\epsilon(\sigma)}
f(u_{\sigma(1)}\odot \squeezedcdots \odot u_{\sigma(a_1)})\cdots 
f(u_{\sigma(n-a_k+1)}\odot \squeezedcdots \odot u_{\sigma(n)})
\end{equation}
where $\sigma$ runs now over 
all $(a_1,\ldots,a_k)$-unshuffles of $n$
elements. The calculation is completed~by 
\[
\exp(f)(1) = e^{f(1)},
\] 
the `actual' exponential of $f(1)$ in $B \hot \frm$.
\end{example}

\section{The category of MV-algebras}

In this section we define the category of MV-algebras over $R$. Firstly,
we introduce morphisms:

\begin{definition}
The space of \emph{(MV-)mor\-phisms} between
\MV-algebras $V'= (V',\mu',\delta',\Delta')$ and $V''=
(V'',\mu'',\delta'',\Delta'')$ is given~by
\[
\ttMV_R(V',V'') := \big\{ f \in \Lin_R^0(V',V'')\; | \;
\Delta'' \circ \exp(f) = \exp(f) \circ \Delta'\big\} .
\]
The categorical composition of $f \in \ttMV_R(V'',V''')$ with $g \in
\ttMV_R(V',V'')$ is defined as
\[
f \diamond g : = \log(\exp (f) \circ \exp(g)).
\]
The unit endomorphism of an \MV-algebra $V$ over $R$ is
$\un_V := \log(\id_{V\hot R})$, which is defined because
$(\id_{V\hot R} - \sfe) \circ \eta = \eta - \eta = 0$, whence
$\id_{V\hot R} - \sfe \in \Lin_R^0(V,V)$.

\end{definition}

\begin{theorem}
  $\ttMV_R$ with the above notion of morphisms forms a category.
\end{theorem}

\begin{proof}
Let us verify the associativity of the categorical composition.
By definition,
\[
(f \diamond g) \diamond h =
\log(\exp (f \diamond g) \circ \exp{h}) =
\log (\exp(f) \circ \exp (g) \circ \exp(h))
\]
while
\[
f \diamond (g
\diamond h) = 
\log(\exp f \circ \exp(g   \diamond h)) =
\log (\exp(f) \circ \exp (g) \circ \exp(h)),
\]
so $(f \diamond g)
\diamond h = f \diamond (g
\diamond h)$ as required.  
To verify the axiom for the categorical identity is also easy; one has
\[
f \diamond \un_V = \log (\exp(f) \circ \id_{V \hot R}) =  \log (\exp(f)) = f.
\]
The identity $\un_V \diamond f = f$ is verified in a similar fashion.
The last thing which remains to be verified is that $\Delta''' \circ
\exp (f \diamond g) = \exp (f \diamond g) \circ \Delta' $ which
follows from
\[
\Delta''' \circ \exp (f \diamond g)  =
\Delta''' \circ \exp (f) \circ  \exp(g)  =
\exp (f) \circ  \exp(g) \circ \Delta' =
\exp (f \diamond g) \circ \Delta' 
\]
\end{proof}

\begin{example}
  Let us show that, as stated in~\cite{sasha's}, the unit endomorphism
  \[
\un_{S(U)} \in \Lin^0_{\bfk}\big(S(U),S(U)\big)
\]
of the $\IBL_\infty$-algebra $S(U)$ recalled in
Example~\ref{Treti_den_pekla} is the projection $\pi_1:S(U) \to U$ to
the space of algebra generators. Since $\pi_1$ is the projection to
$U$, it follows from  formula~(\ref{vedro}) for the iterated
diagonal that $
\pi_1^{\ot k} \circ \delta^{[k-1]}(u_1 \odot \cdots \odot u_n) \not =0 $
only when $k=n$, in which case~(\ref{vedro})~gives
\[
\mu^{[k-1]}\circ \pi_1^{\ot k}\circ \delta^{[k-1]} (u_1 \odot \cdots \odot u_n) = 
\begin{cases}
{n!\ (u_1 \odot \cdots \odot u_n) }&{\hbox {if $n = k$ and}}
\\
0&{\hbox{otherwise.}}
\end{cases}
\]
This readily implies that $\exp(\pi_1) = \id_{S(U)}$ i.e.\
that $\un_{S(U)} =\pi_1 = \log(\id_{S(U)})$ as claimed. 

We recommend as an exercise to verify that  $\pi_1 =
\log(\id_{S(U)})$ directly. It turns out that this equation leads to an
interesting combinatorial formula for the unshuffles.
\end{example}

\begin{remark}[Geometric interpretation]
\label{prepsano_v_Tokyo}
  We can interpret a morphism in $\ttMV_R(V',V'')$ geometrically as a
  generalized, as in Remark~\ref{psano_v_St_Paule}, morphism $X'' \to
  X'$ of \MV-manifolds. Dually, we can think of it as a morphism
  ${X'^*} \to {X''^*}$ of dual \MV-manifolds.
\end{remark}

\begin{example}
Notice that if the reduced diagonal in $V'$ is trivial, 
\[
\exp(f) \equiv \sfe + f   \ \hbox { mod } V''\hot\, \frm,    
\hbox { and } \log(\sfe + g) =  g  \ \hbox { mod } V''\hot\, \frm.
\] 
The category $\ttMV_\bfk$ has a full subcategory consisting of
\MV-algebras over $\bfk$
with trivial reduced diagonals. The composition rule in this
subcategory is given by
\begin{align*}
f \diamond g & = \log(\exp (f) \circ \exp(g)) = 
\log ((\sfe + f) \circ (\sfe + g))
\\
& = \log(\sfe \circ \sfe + \sfe \circ g + f \circ \sfe + f \circ g) = 
 f \circ g +  \sfe \circ g + f \circ \sfe.
\end{align*}
We used the fact that $\sfe \circ \sfe = \sfe$ by~(\ref{eq:2}). It is
an instructive exercise to verify that the categorical
unit endomorphism is $\id - \sfe$. Notice that the `expected' unit $\id$
is not even an element of $\Lin_R^0(V,V)$. 

Restricting to an even smaller subcategory whose morphisms $f \in
\ttMV_{R} (V',V'')$ satisfy the stronger condition
\[
f\circ \eta' = 0, \ \epsilon'' \circ f = 0,
\]
the composition rule $f \diamond g$ becomes the standard composition
of morphisms.
\end{example}

\begin{example}
  The category $\ttMV_{\bfk}$ contains the subcategory $\L_\infty$
  whose objects are ${\rm L}_\infty$-algebras recalled in
  Example~\ref{uz_ctyri_hodiny} and morphisms are maps $f :S(U') \to
  S(U'')$ such that
\[
f(1) = 0, \ \Delta'' \circ \exp(f) = \exp(f) \circ \Delta', \ \hbox {
  and } \ \Im(f) \subset U''.
\]
Such a map automatically belongs to
$\Lin_{\bfk}^0\big(S(U'),S(U'')\big)$. We leave as an interesting
exercise to prove that
\[
\exp(f) : S(U') \to S(U'') 
\]
is the unique coextension of $f$ into a morphism of counital
coalgebras. We conclude that $\L_\infty$ is isomorphic to the category
of ${\rm L}_\infty$-algebras and their (weak) ${\rm
  L}_\infty$-mor\-phisms~\cite[Remark~5.3]{lada-markl:CommAlg95}.
\end{example}

\begin{example}
\label{ibl}
 Let us  consider
 $\IBL_\infty$-algebras
  recalled in Example~\ref{Treti_den_pekla} with
  $\bfk[[\hbar]]$-linear maps
\[
f : S(U')[[\hbar]] \to S(U'')[[\hbar]]
\]
of the form
\begin{equation}
\label{Jaruska_zitra_podepise_smlouvu}
f = f^{(1)} + \hbar f^{(2)} + \hbar^2 f^{(3)} + \cdots
\end{equation}
such that 
\begin{gather}
  \nonumber f^{(1)}(1) = 0, \quad \Delta'' \circ \exp(f) = \exp(f)
  \circ \Delta', \quad \hbox{ and }\\
\label{cl-cfl}
\bigoplus_{n > k} S^n(U') \subset \Ker (f^{(k)}).
\end{gather}
In Corollary~\ref{cor} below we prove that the above structure forms a
subcategory $\ttIBL_\infty$ of the category $\ttMV_{\bfk[[\hbar]]}$ of
MV-algebras over $\bfk[[h]]$, c.f.~also 
\cite[\S5]{cl}  and
\cite[Definition 2.8]{cfl}. One can consider a version of this
subcategory with the condition \eqref{cl-cfl} replaced with a ``dual''
condition:
\[
\Im(f^{(k)}) \subset \bigoplus_{1 \leq n \leq k}S^n(U'').
\]
This modified subcategory of $\ttMV_{\bfk[[\hbar]]}$ may be called
the category of $\IBL_\infty$-algebras in the sense
of~\cite[\S4.3]{munster-sachs}.
\end{example}

\begin{example}[Cieliebak-Latschev \cite{cl}]
\label{bv}
A $\BV_\infty$-mor\-phism from a $\BV_\infty$-algebra $(S(U), \Delta')$
of Example~\ref{Treti_den_pekla} to a $\BV_\infty$-algebra $(A,
\Delta'')$ of Example~\ref{sec:mv-algebras-1} is an \MV-mor\-phism given
by a $\bfk[[\hbar]]$-linear map
\[
f : S(U)[[\hbar]] \to A[[\hbar]]
\]
of the form
\[
f = f^{(1)} + \hbar f^{(2)} + \hbar^2 f^{(3)} + \cdots,
\]
such that 
\begin{gather*}
  \nonumber f^{(1)}(1) = 0, \quad \Delta'' \circ \exp(f) = \exp(f)
  \circ \Delta', \quad \hbox{ and }\\
\bigoplus_{n > k} S^n(U) \subset \Ker (f^{(k)}).
\end{gather*}
This is a generalization of the notion of an $\IBL_\infty$-mor\-phism of
the type \eqref{cl-cfl}.
\end{example}

We are going to define a product $V' \oslash V''$ of 
two \MV-algebras $V' =  (V',\mu',\delta',\Delta')$ and 
$V'' = (V'',\mu'',\delta'',\Delta'')$ over $R$
 as follows. Its underlying graded vector space is 
$V' \ot V''$ and the structure operator is
given as $\Delta' \ot_R \id + \id \ot_R \Delta''$. The multiplication
is defined in the standard way: 
\[
(v'_1 \ot v''_1) \cdot (v'_2 \ot v''_2)
:= (-1)^{\abs{v''_1}\abs{v'_2}} v'_1 \cdot v'_2 \ot v''_1 \cdot
v''_2, \ v_1',v_2' \in V' , \ v_1'',v_2'' \in V'',
\] 
with the unit given by 
the map $\eta' \ot \eta'' :\bfk \cong \bfk \ot \bfk \to V'
\ot V''$.  The comultiplication is
defined~as 
\[
\delta(v' \ot v'') := \tau_{23} \big(\delta'(v') \ot
\delta''(v'')\big), \ v' \in V',\ v'' \in V'',
\]
where $\tau_{23}$ permutes the second and the third
factors with the Koszul sign and, finally, 
$\epsilon' \ot \epsilon'' :V' \ot V'' \to
\bfk \ot \bfk \cong \bfk$ is the
counit. The $\oslash$-product of morphisms $f \in \MV_R(V'_1,V'_2)$
and $g \in \MV_R(V''_1,V''_2)$ is given by the formula
\[
f \oslash g := \log \big( \exp (f) \ot_R \exp(g) \big).
\]

\begin{proposition}
The $\oslash$-product equips $\ttMV_{R}$ 
with a symmetric monoidal structure.
\end{proposition}

\begin{proof}
Direct verification.
\end{proof}

\begin{proposition}
\label{Stihnu_vse_pred_odletem_do_Prahy?}
  The category $\ttMV_{R}$ of \MV-algebras over $R$ is isomorphic to
  the category $\widetilde{\ttMV}_{R}$ with the same objects and
  morphisms
\[
\widetilde{\ttMV}_{R}(V', V'') := \big\{\varphi \in \Lin_R(V', V'')\ | \
\varphi \circ
  \eta' \equiv \eta'' \!\! \mod V'' \hot\, \frm \
\mbox { and }\ \Delta'' \circ \varphi
  = \varphi \circ \Delta'\big\}.
\]
The categorical composition is the usual composition
  of maps, and the unit ${\tt 1}_V \in  \widetilde{\ttMV}_{R}(V, V)$
  is the identity
  $\id : V \to V$.
\end{proposition}

\begin{proof}
  The isomorphism between $\ttMV_{R}$ and $\widetilde{\ttMV}_{R}$ 
is given by the mutually inverse functors identical
  on objects and taking a morphism $f \in \ttMV_R(V',V'')$ to $\exp
  (f) \in \widetilde{\ttMV}_{R}(V',V'')$ and 
$\varphi \in \widetilde{\ttMV}_{R}(V',V'')$ to $\log(\varphi) \in 
\ttMV_R(V',V'')$.  Let us check that this isomorphism is
  well-defined.
  
Notice first that for $\varphi \in \Lin_R(V', V'')$, the condition
$\varphi \circ \eta' \equiv \eta'' \mod V'' \hot\, \frm$ is equivalent
to the condition $(\varphi - {\tt e})\circ \eta' \equiv 0 \mod V''
\hot\, \frm$. To see it, recall that $ {\tt e} = \eta''
\circ \epsilon '$,~thus
\[
(\varphi - {\tt e})\circ \eta' = \varphi\circ \eta' - \eta'' \circ
\epsilon ' \circ \eta' =  \varphi\circ \eta' - \eta''.
\]
Likewise, for  $f \in \Lin_R(V', V'')$,  $f \circ
  \eta' \cong 0 \mod V'' \hot\, \frm$ is equivalent to  
$(f + {\tt e})\circ \eta'  \equiv \eta'' \mod V'' \hot\,  \frm$.

It is now easy to verify, using the definitions of $\exp$ and $\log$, 
that $\exp (f)$ indeed belongs to $\widetilde{\ttMV}_{R}(V', V'')$ and
$\log(\varphi)$ to  $\ttMV_R(V', V'')$. The fact that the above
correspondence converts the $\diamond$-composition to the usual one is clear.
\end{proof}

\def\chain{{\tt Chn}_R^\circ}
\begin{corollary}
  The category $\ttMV_R$ of \MV-algebras over $R$ is equivalent to the
  category $\chain$ of \emph{pointed complexes} over $R$ whose objects
  are graded vector spaces $V$ with a continuous degree $+1$
  $R$-linear differential $\Delta: V \hot R \to V \hot R$ and a
  $\bfk$-linear monomorphism $\eta: \bfk \to V$ such that
  $\Delta\circ \eta = 0$. A morphism between $V'$ and $V''$ is a chain
  map $\varphi \in \Lin_R(V', V'')$ such that
  $\varphi \circ \eta' \equiv \eta'' \mod V'' \hot \frm$.
\end{corollary}

\begin{proof}
By Proposition~\ref{Stihnu_vse_pred_odletem_do_Prahy?}, it is enough
to establish an equivalence between the categories
$\widetilde{\ttMV}_{R}$ and $\chain$.
Let us construct mutual weak inverses $\Box : \widetilde{\ttMV}_{R}
\to \chain$ and $F : \chain \to \widetilde{\ttMV}_{R}$.

On objects, the functor $\Box$ forgets everything except 
the structure operator $\Delta$ and
the unit map $\eta$.  For $(V,\Delta,\eta)
\in \chain$ choose a right inverse $\epsilon$ of $\eta$ and define
$F(V,\Delta,\eta)$ the supertrivial MV-algebra as in
Example~\ref{supertrivial}. 
Notice that, for $V',V'' \in \widetilde{\ttMV}_{R}$,
\[
\chain( \Box V', \Box V'') = \widetilde{\ttMV}_{R}(V',V'')
\]
and, likewise, for $V',V'' \in \chain$,
\[
 \widetilde{\ttMV}_{R}(F V', F V'') =\chain(V',V'').
\]
We define $\Box$ and $F$ to be the identities on morphisms. It is
simple to verify that we have constructed mutual weak inverses.
\end{proof}

\begin{remark}
\label{sec:morphisms-1}
The definition of the category of \MV-algebras can be modified. For
instance, we may leave the conilpotency of $\delta$ out, but instead
of~(\ref{eq:1}) require that $\Im(f) \in V'' \hot\, \frm$. Likewise, we
need not require $R$ to be complete, but then~(\ref{eq:1}) must be
replaced by $f \circ \eta' = 0$.  In both cases the above
constructions remain valid.
We may also require $\epsilon \circ \Delta = 0$, drop the condition
$\Delta(1)=0$, or require both conditions simultaneously.
We may
allow $R$ to be differential graded, which could be useful in some contexts.
\end{remark}

\section{Generalizations to other algebra types} 

MV-algebras were defined as spaces that are simultaneously commutative 
associative algebras and  cocommutative 
coassociative coalgebras. In this mildly speculative section we
discuss possible generalizations to
structures  other  than commutative associative (co)-algebras.
We will assume basic
knowledge of operads as it can
be gained for example from~\cite{markl-shnider-stasheff:book}. 

As preparation we view the
exponential~(\ref{Zitra_budeme_s_Jaruskou_pit_portske.}) from a
different angle. Namely, we describe the  isomorphism
\[
\exp - \ \sfe :  \Lin^0_R\big(V',V''\big) 
\stackrel\cong\longrightarrow  \Lin^0_R\big(V',V''\big),
\]
which was the core of our construction, in terms of universal algebra, 
assuming for simplicity that $R = \bfk$. Let us denote by
$\cS(V'')$ the symmetric algebra $S(V'')$ considered as a
coalgebra with the standard coalgebra structure. Since 
$\cS(V'')$ with the projection $\cS(V'') \to V''$ realizes
the cofree conilpotent coassociative cocommutative
coalgebra cogenerated by $V''$ 
\cite[Example~II.3.79]{markl-shnider-stasheff:book}, each  $f : V' \to V'' \in
\Lin^0_\bfk\big(V',V''\big)$ uniquely coextends to a coalgebra map
$u_f : V' \to \cS(V'')$.\footnote{Here our assumption of the
  conilpotency of $V'$ resurfaces again.} 

On the other hand, the  multiplication of $V''$ determines a
linear map \hbox{$m :  S(V'') \to V''$}. Expressing the coextension
$u_f : V' \to \cS(V'')$ using e.g.~formula~(3.66) in Section II.3.7 of 
\cite{markl-shnider-stasheff:book} with $\oP$ the operad for
commutative associative algebras,  we easily see that
$\exp(f)\!  - \sfe$ equals the
composition 
\begin{equation}
\label{Vcera_jsem_lyzoval_na_rozmoklem_snehu_v_Ratajich.}
V' \stackrel{u_f} \longrightarrow \cS(V'') 
\stackrel{{\rm can}}\longrightarrow S(V'') \stackrel m\longrightarrow V''
\end{equation}
in which
\[
{\rm
  can} : \cS(V'') \stackrel \cong\longrightarrow S(V'')
\]
is the identity
of the underlying graded vector spaces.

Let us try to generalize the
composed map~(\ref{Vcera_jsem_lyzoval_na_rozmoklem_snehu_v_Ratajich.})
to the case when  $V'$ is a $\oP$-coalgebra
and $V''$ a $\oQ$-algebra, for some $\bfk$-linear 
operads  $\oP$ and $\oQ$. We may assume from the very beginning that
$\oP$ has finite-dimensional pieces, as most operads relevant for
physical applications have this property.  We
certainly have again the  canonical morphisms $u_f$ and $m$ in the sequence
\begin{equation}
\label{dvakrat_za_sebou}
V' \stackrel{u_f} \longrightarrow \cF(V'') 
\stackrel{{?}}\longrightarrow \Fr(V'') \stackrel m\longrightarrow V''
\end{equation}
in which $\cF(V'')$ is the cofree
conilpotent $\oP$-coalgebra on $V''$ and $\Fr(V'')$ the free
$\oQ$-algebra on $V''$. The only datum which is not automatic
is an isomorphism 
\begin{equation}
\label{otaznik}
? :
\cF(V'') \longrightarrow  \Fr(V'').
\end{equation}
Its existence must therefore be
accepted as an assumption, i.e.~the operads $\oP$ and $\oQ$  
must be such that the graded spaces 
$\cF(V'')$ and $
\Fr(V'')$ are isomorphic via an isomorphism natural in $V''$.

To formulate this assumption solely in terms of the operads $\oP$
and $\oQ$, we invoke from \cite[Definitions~II.1.24 and
II.3.74]{markl-shnider-stasheff:book} the formulas
\[
\cF(V'') = \bigoplus_{n\geq 1} \big(\oP(n)^*\otimes {V''}^{\otimes
  n}\big)^{\Sigma_n}
\ \hbox { and } 
\
\Fr(V'') = \bigoplus_{n\geq 1} \oQ(n)\otimes_{\Sigma_n} {V''}^{\otimes
  n},
\]
where $\oP(n)^*$ is the linear dual of the vector space  $\oP(n)$.
It is easy to see now that if a functorial isomorphism 
in~(\ref{otaznik}) exists  then one has for each $n \geq 1$ an
isomorphism
\begin{equation}
\label{Univime}
\oP(n)^* \cong \oQ(n).
\end{equation}
It must moreover, very crucially, be `nice' and
explicit enough\footnote{`Nice'  means in particular that the
  isomorphism explicitely relates the cooperad structure of $\oP^*$
  with the operad structure of $\oQ$. Paragraph~2.5
of \cite{markl:ab} shall give a more
concrete idea what we mean by it  when $\oP
= \oQ = \Lie$, the operad for Lie algebras.} 
so that we could express the
composition~(\ref{dvakrat_za_sebou})
by a formula involving the convolution product in
$\Lin^0_\bfk\big(V',V''\big)$. 

The existence of~(\ref{Univime}) is already very restrictive. Since we
assumed that the pieces of the operad
$\oP$ are finite-dimensional, it implies that $\oP(n) \cong \oQ(n)$ for
each $n$, so the generating series of the operads $\oP$ and $\oQ$ are
the same. We do not know about any couple of different operads
relevant for applications with the same generating series. We are thus led
to the assumption $\oP = \oQ$, supported by the
natural requirement of  essential self-duality of the definition of
MV-algebras.  

Finding interesting operads $\oP$ admitting a nice
  isomorphisms  $\oP(n)^* \cong \oP(n)$, $n \geq 1$, is however a
  difficult task. For
instance, an explicit canonical isomorphism
\[
\Lie(n)^* \cong \Lie(n) 
\]
for the operad $\Lie$ governing Lie
algebras
is known only for small $n$, and finding one is closely related to the
problem of Eulerian idempotents, see the discussion in \S2.5
and Remark~2.9 of \cite{markl:ab}.

On the other hand, a nice canonical isomorphism as in~(\ref{Univime})
always exists when $\oP = \oQ$ are
$\bfk$-linearizations of an operad $\mathbf p$ 
defined in the category of sets, as
then $\oP(n)$ has for each $n \geq 1$ a canonical basis spanned by
the elements of ${\mathbf p}(n)$. There are two prominent examples of
this situation. The first one is $\oP = \oQ = \Com$, the operad for
commutative associative algebras which is the linearization of the
terminal set-operad. The corresponding MV-algebras are the ones
discussed in this paper.

The second outstanding example is $\oP = \oQ = \Ass$, the operad for
associative algebras which is the linearization of the terminal
non-$\Sigma$ set-operad. The corresponding theory should be that of an 
$A_\infty$-version of MV-algebras. We expect that it has a similar
flavor as the commutative one, with the notable difference that
the exponential~(\ref{Zitra_budeme_s_Jaruskou_pit_portske.}) shall be
replaced by the series
\[
\sfe + f + f^2 + f^3 + \cdots  = (\sfe - f)^{-1}
\]
and the logarithm by its functional inverse
$(g - \sfe) g^{-1}$.

Let us close this section by a remark about the convolution
product. In general it equips, 
for $V'$ a $\oP$-coalgebra and $V''$ a $\oQ$-algebra, 
$\Lin^0_R\big(V',V''\big)$ only with a structure of a $(\oP \otimes
\oQ)$-algebra. A special feature of the cases  $(\oP,\oQ) =
(\Com,\Com)$ or $(\Ass,\Ass)$ is that both $\Com$ and $\Ass$ are Hopf
operads \cite[Definition~II.3.135]{markl-shnider-stasheff:book}, i.e.\ the ones equipped with the diagonals
\[
\Com \longrightarrow \Com \ot \Com \ \hbox { and } \ \Ass \longrightarrow \Ass
\ot \Ass,
\] 
which make the space  $\Lin^0_R\big(V',V''\big)$ actually a
commutative associative algebra, respectively an associative
algebra. Since each operad which is a linearization of a set-theoretic
one is a Hopf operad~\cite[Proposition~11]{markl-remm:JA06}, such a property of the convolution product
holds for all operads of this type.

We conclude that sensible generalizations of MV-algebras
may exist for couples of the form $(\oP,\oP)$, where $\oP$ is a
linearization of a set-theoretic operad.  We however think that
working out the details  would make sense 
only when a relevant motivating example appears.  

\section{A composition formula and $\IBL_\infty$-algebras}

Let us consider morphisms
$g \in \Lin^0_R\big(S(U'), S(U'')\big)$ and 
$f \in \Lin^0_R\big(S(U''),S(U''')\big)$, where $S(U'), S(U'')$ and $S(U''')$ 
are symmetric algebras with the standard coalgebra structures. The aim
of this section is to give an explicit formula for 
\begin{equation}
\label{Jaruska_mi_poslala_Mikulase}
f \diamond g  = \log\big(\exp (f) \circ \exp(g)\big) \in 
\Lin^0_R\big(S(U'),S(U''')\big).
\end{equation}
Further, using this formula, we prove that $\IBL_\infty$-algebras with
morphisms~(\ref{Jaruska_zitra_podepise_smlouvu}) form a subcategory of
$\ttMV_{\bfk[[\hbar]]}$.

\def\epi{\twoheadrightarrow}
\def\vlra{{\hbox{$-\hskip-1mm-\hskip-2mm\longrightarrow$}}}
Let us formulate some preparatory observations.  
Each $R$-linear map
\[
\hbox{$h : S(V')\hot R \to S(V'')\hot R$}
\] 
determines a family
\[
h^m_n : S^n(V') \to S^m(V'')\hot R,\ m,n \geq 0,
\]
with $h^n_m$ the composition 
\[
S^n(V') \hookrightarrow S(V') 
\stackrel{ h|_{S(V')}}\vlra S(V'') \hot R \epi
S^n(V') \hot R ,\footnote{We use the convention that $S^0(V') =
  S^0(V'') = \bfk$.}
\]
where $\hookrightarrow$ resp.~$\epi$ is the canonical inclusion resp.~
the canonical projection. Vice versa, each family $\{h^m_n\}_{m,n \geq
  0}$ as above such that the sum
\[
h|_{S(V')}(x)
:=\sum_{m \geq 0} h^m_n(x) 
\]
converges in $S(V'')\hot R$ for each fixed $n \geq 0$ and 
$x \in S^n(V')$, determines an $R$-linear map $h : S(V')\hot R \to S(V'')\hot
R$.
\def\flicek#1#2{({}^{#2}_{#1})}
We will call $h^m_n$ the
{\em $\flicek {{\, n}}m$-component\/} of $h$. 
We will describe $f \diamond g$ in terms of its $\flicek {{\, n}}m$-components.

For natural numbers $k,l$ and \nni{}s 
$r,s_1,\ldots,s_l$, $j_1,\ldots,j_k $ such that
\[
j_1 + \cdots + j_k = s_1 + \cdots  + s_l = r
\]
we define
\begin{equation}
\label{beh}
\Psi_{j_1,\ldots,j_k}^{s_1,\ldots,s_l} : S^{j_1}(U'') \ot \cdots \ot
S^{j_k}(U'')
\longrightarrow  S^{s_1}(U'') \ot \cdots \ot
S^{s_l}(U'')
\end{equation}
to be the $\bfk$-linear map that sends
\[
(u''_1 \odot \cdots \odot
u''_{j_1}) \ot \cdots
\ot (u''_{r-j_k+1} \odot \cdots \odot
u''_{r}) \in S^{j_1}(U'') \ot \cdots \ot
S^{j_k}(U'')
\]
into the sum
\begin{equation}
\label{foudling}
\sum_{\kappa \in \unsh_{s_1,\ldots,s_l}}\epsilon(\kappa)
(u''_{\kappa(1)} \odot \cdots \odot
u''_{\kappa(s_1)}) \ot \cdots
\ot (u''_{\kappa(r-s_l+1)} \odot \cdots \odot
u''_{\kappa(r)}) \in  S^{s_1}(U'') \ot \cdots \ot
S^{s_l}(U'')
\end{equation}
over the set $\unsh_{s_1,\ldots,s_l}$ of all $(s_1,\ldots,s_l)$-unshuffles
$\kappa$ of $r = s_1+\cdots + s_l$ elements. It is easy to check
that~(\ref{foudling}) is well-defined, i.e.\
that~(\ref{foudling}) is invariant under permutations of generators 
inside the
groups 
\begin{equation}
\label{v_pracovne_v_Minneapolis}
\{u''_1,\ldots,u''_{j_1}\},\ldots,\{u''_{r-j_k+1},\ldots,u''_{r}\}.
\end{equation}

We associate to each $\kappa \in \unsh_{s_1,\ldots,s_l}$
in~(\ref{foudling}) a graph $\Gamma$ with two types of vertices, the
upper and lower ones. The upper ones are labelled by $1,\ldots,l$, the
lower ones by $1,\ldots,k$.  The upper vertex labelled by $b \in
\{1,\ldots,l\}$ is connected to the lower vertex labelled by $a\in
\{1,\ldots,k\}$ if and only if $\kappa$ maps some element of the $b$th
segment of the decomposition
\[
\{1,\ldots,r\} = \{1,\ldots,s_1\} \cup \cdots \cup \{r-s_l+1,\ldots,r\}
\]
to some element of the $a$th segment of the decomposition
\[
\{1,\ldots,r\} = \{1,\ldots,j_1\} \cup \cdots \cup \{r-j_k+1,\ldots,r\}.
\]

Let $c(\kappa)$ be the number of connected components of $\Gamma$. We
will call $c(\kappa)$ the {\em connectivity\/} of $\kappa$ and say
that~$\kappa$ is {\em connected\/} if $c(\kappa)=1$.  Denote by
${}^c\Sigma^{j_1,\ldots,j_k}_{s_1,\ldots,s_l}$ the subset of all $\kappa
\in \unsh_{s_1,\ldots,s_l}$ of connectivity  $c$ and let finally
\[
{}^c\Psi_{j_1,\ldots,j_k}^{s_1,\ldots,s_l}: S^{j_1}(U'') \ot \cdots \ot
S^{j_k}(U'')
\to  S^{s_1}(U'') \ot \cdots \ot
S^{s_l}(U'')
\]
be the map defined as $\Psi_{j_1,\ldots,j_k}^{s_1,\ldots,s_l}$ but 
with the sum~(\ref{foudling}) restricted to connected
$\kappa \in{}^c\Sigma^{j_1,\ldots,j_k}_{s_1,\ldots,s_l}$.

\begin{proposition}
The maps
${}^c\Psi_{j_1,\ldots,j_k}^{s_1,\ldots,s_l}$, $c \geq 1$, are
well-defined and 
\begin{equation}
\label{dnes}
\Psi_{j_1,\ldots,j_k}^{s_1,\ldots,s_l} =
{}^1\Psi_{j_1,\ldots,j_k}^{s_1,\ldots,s_l} +
{}^2\Psi_{j_1,\ldots,j_k}^{s_1,\ldots,s_l} +  
{}^3\Psi_{j_1,\ldots,j_k}^{s_1,\ldots,s_l}+\cdots.
\end{equation}
\end{proposition}

\begin{proof}
Since composing with permutations inside the
groups~(\ref{v_pracovne_v_Minneapolis}) does not change the connectivity
of $\kappa$, each individual
${}^c\Psi_{j_1,\ldots,j_k}^{s_1,\ldots,s_l}$ is
well-defined. Formula~(\ref{dnes}) is obvious.
\end{proof}

In what follows we use the same symbols both for the
maps~(\ref{dnes}) and for their $R$-linear extensions. Therefore,
for instance, ${}^1\Psi_{j_1,\ldots,j_k}^{s_1,\ldots,s_l}$ will also
denote an $R$-linear map
\begin{equation}
  \label{Jaruscina_dymka}
{}^1\Psi_{j_1,\ldots,j_k}^{s_1,\ldots,s_l}:
\big[S^{j_1}(U'') \ot \cdots \ot
S^{j_k}(U'')\big] \hot R
\to \big[ S^{s_1}(U'') \ot \cdots \ot
S^{s_l}(U'')\big] \hot R.
\end{equation}

Let $g^m_n: S^n(U') \to S^m(U'')\hot R$ be,  for $m,n \geq 0$, the $\flicek {{\,
    n}}m$-components of $g$ and let  \[
f^m_n :
S^n(U'') \hot R \to S^m(U''') \hot R
\] 
be the $R$-linear extensions of
the $\flicek {{\, n}}m$-components of $f$. For given natural numbers
$k,l$ and \nni{}s $i_1,\ldots,i_k, t_1,\ldots,t_l$ 
we define an auxiliary map
\[
(f\diamond g)_{i_1,\ldots,i_k}^{t_1,\ldots,t_l}: S^{j_1}(U'') \ot \cdots \ot
S^{j_k}(U'')
\longrightarrow
\big[S^{s_1}(U'') \ot \cdots \ot S^{s_l}(U'')\big] \hot R
\]
as the sum
\begin{equation}
\label{Jeste_vcera)_jsem_byl_na_lyzickach.}
\sum \frac 1{k!\ l!}
(f_{s_1}^{t_1} \ot \cdots \ot f_{s_l}^{t_l}) 
\circ {}^1\Psi_{j_1,\ldots,j_k}^{s_1,\ldots,s_l}
\circ (g_{i_1}^{j_1} \ot \cdots \ot g_{i_k}^{j_k})
\end{equation}
over all \nni{}s $s_1,\ldots,s_l,j_1,\ldots,j_k $ such that $s_1 +
\cdots + s_l = j_1 + \cdots + j_k$, where
${}^1\Psi_{j_1,\ldots,j_k}^{s_1,\ldots,s_l}$ is the $R$-linear
extension~(\ref{Jaruscina_dymka}) of the connected component of the
map $\Psi_{j_1,\ldots,j_k}^{s_1,\ldots,s_l}$ in~(\ref{beh}).  In the
proof of Theorem~\ref{sculpture_garden} below we
write~\eqref{Jeste_vcera)_jsem_byl_na_lyzickach.} 
in the form
\begin{equation}
\label{mikulas}
\sum \frac 1{k!\ l!}\left(
\begin{array}{c}
f_{s_1}^{t_1} \ot \cdots \ot f_{s_l}^{t_l}
\\
{}^1\Psi_{j_1,\ldots,j_k}^{s_1,\ldots,s_l}
\\
g_{i_1}^{j_1} \ot \cdots \ot g_{i_k}^{j_k}
\end{array}
\right)
\end{equation}
which would enable us to squeeze formula~(\ref{Main}) into a line of
finite length.
Finally, for $m,n > 0$ define a $\bfk$-linear map 
\[
(f\diamond g)^m_n : S^n(U') \to S^m(U''')\hot R
\]
by the formula
\begin{align}
\label{hodinky}
(f\diamond g)^m_n(u'_1\!\odot \cdots \odot\! u'_n&):=
\\
\nonumber 
\sum\mu^{[l-1]} \epsilon(\sigma) &
(f\diamond g)_{i_1,\ldots,i_k}^{t_1,\ldots,t_l} (u'_{\sigma(1)}\! \odot \cdots \odot\! u'_{\sigma(i_1)}) \ot \cdots \ot 
(u'_{\sigma(n-i_k+1)}\! \odot  \cdots \odot\! u'_{\sigma(n)}),
\end{align}
where the summation runs over all
natural numbers $k,l$, \nni{}s 
$i_1,\ldots,i_k$, $t_1,\ldots,t_l$ such that 
\[
i_1 +
\cdots i_k = n \ \mbox { and }\
t_1 + \cdots + t_l = m,
\] 
and over all $(i_1,\ldots,i_k)$-unshuffles $\sigma \in
\unsh_{i_1,\ldots,i_k}$. In~(\ref{hodinky}),
\[
\mu^{[l-1]}
: \big[  S^{t_1}(U''') \ot \cdots \ot S^{t_l}(U''')\big] 
\hot R \longrightarrow
S^m(U''')\hot R
\]
is the $R$-linear extension of the 
multiplication in $S(U''')$ iterated $(l\!-\!1)$-times.

\begin{theorem}
\label{sculpture_garden}
The $\flicek {\, n}m$-part of the composition 
$f \diamond g \in \Lin^0_R\big(S(U'),S(U''')\big)$ 
is given by formula~(\ref{hodinky}).
\end{theorem}

\begin{proof}
Our strategy will be to show that the exponential of the map whose
$\flicek {\, n}m$-parts are given by~(\ref{hodinky}) equals \hbox{$\exp(f)
\!\circ \exp(g)$}. 
Using~(\ref{Zitra_vylet_na_kole_podel_Missisipi}) we find
the following expression for the
$\flicek {{\, n}}m$-components of the composition $\exp(f)\! \circ\! \exp(g)$:
\begin{align*}
\big(\exp(f)\! \circ \!\exp(g)\big)^m_n(u'_1\!\odot \cdots \odot\! u'_n)&=
\\
\nonumber 
\sum\mu^{[l-1]}\epsilon(\sigma)\big(\exp(f) 
\circ \exp(g)&\big)_{i_1,\ldots,i_k}^{t_1,\ldots,t_l} 
(u'_{\sigma(1)}\! \odot \cdots \odot\! u'_{\sigma(i_1)})\! \ot \cdots \ot \!
(u'_{\sigma(n-i_k+1)}\! \odot  \cdots \odot\! u'_{\sigma(n)}),
\end{align*}
where 
\begin{equation}
\label{Fesacek}
\big(\exp(f) 
\circ \exp(g)\big)_{i_1,\ldots,i_k}^{t_1,\ldots,t_l} =:
\sum \frac 1{k!\ l!}\left(
\begin{array}{c}
f_{s_1}^{t_1} \ot \cdots \ot f_{s_l}^{t_l}
\\
\Psi_{j_1,\ldots,j_k}^{s_1,\ldots,s_l}
\\
g_{i_1}^{j_1} \ot \cdots \ot g_{i_k}^{j_k}
\end{array}\right)
\end{equation}
with the summation as in~\eqref{mikulas}. The crucial difference 
against~(\ref{mikulas}) is however that~(\ref{Fesacek})
involves the entire 
$\Psi_{j_1,\ldots,j_k}^{s_1,\ldots,s_l}$ not only its connected part.

Now the theorem follows from the principle standard in the theory of
Feynman diagrams\footnote{See e.g.~\cite[p.~119]{MR603127}.} that the logarithm
singles out connected components 
of graphs or, equivalently, that the exponential assembles graphs from
their connected components. 
Let us explain what this principle says in our case. 
Calculating the exponential
of $f \diamond g$ using formula~(\ref{hodinky}) involves, instead
of~(\ref{Fesacek}), expressions like
\begin{equation}
\label{Main}
\sum  \frac 1{u!k_1!\ l_1! \cdots k_u!l_u!}
\left(
\begin{array}{c}
\rule{0pt}{1em}
\redukce{$f_{s^1_1}^{t^1_1} \ot \cdots \ot f_{s^1_{l_1}}^{t^1_{l_1}}$}
\\
\rule{0pt}{1em}
\redukce{${}^1\Psi_{j^1_1,\ldots,j^1_{k_1}}^{s^1_1,\ldots,s^1_{l_1}}$}
\\
\rule{0pt}{1em}
\redukce{$g_{i^1_1}^{j^1_1} \ot \cdots \ot g_{i^1_{k_1}}^{j^1_{k_1}}$}
\end{array}
\right)
\ot \cdots \ot
\left(
\begin{array}{c}
\rule{0pt}{1em}
\redukce{$f_{s^u_1}^{t^u_1} \ot \cdots \ot f_{s^u_{l_u}}^{t^u_{l_u}}$}
\\
\rule{0pt}{1em}
\redukce{${}^1\Psi_{j^u_1,\ldots,j^u_{k_u}}^{s^u_1,\ldots,s^u_{l_u}}$}
\\
\rule{0pt}{1em}
\redukce{$g_{i^u_1}^{j^u_1} \ot \cdots \ot g_{i^u_{k_u}}^{j^u_{k_u}}$}
\end{array}
\right)
\end{equation}
with some $u \geq 1$, $l_1 + \cdots + l_u = l$, $k_1 + \cdots + k_u =
k$, and
\[
t^1_1+\cdots + t^1_{l_1} + \cdots +t^u_1+\cdots + t^u_{l_u} = m\
\hbox { and } \
i^1_1+\cdots + i^1_{k_1} + \cdots +i^u_1+\cdots + i^u_{k_u} = n.
\]

In fact, formula~(\ref{Main}) 
expresses the right hand side of~(\ref{Fesacek}) via
products of its
`connected' components. The seeming discrepancy between the
coefficients
\begin{equation}
\label{pozitri_koloquium}
\frac1{k!l!} \ \mbox { in~(\ref{Fesacek}) and}
\
\frac 1{u!k_1!\ l_1! \cdots k_u!l_u!}  \hbox{ in~(\ref{Main})} 
\end{equation}
is explained as follows. While in~(\ref{Fesacek}) the connected
components are entangled, in~(\ref{Main}) they are separated. However,
thanks to the commutativity of symmetric algebras, 
the entangled components can be separated
using suitable permutations on the output and input sides. It is
straightforward though tedious to check that, with the
coefficients~(\ref{pozitri_koloquium}),  the corresponding terms appear
with the same multiplicity.
\end{proof}

\begin{corollary}
\label{cor}
  $\IBL_\infty$-algebras with
  morphisms~\eqref{Jaruska_zitra_podepise_smlouvu} form a subcategory
  $\ttIBL_\infty$ of $\ttMV_{\bfk[[\hbar]]}$.
\end{corollary}

\begin{proof}
It is simple to verify that a map as
in~\eqref{Jaruska_zitra_podepise_smlouvu} satisfies~\eqref{cl-cfl} 
if and only if its $\flicek {{\, n}}m$-component $f^m_n$ is divisible
by $\hbar^{n-1}$ for each $m \geq 0$ and $n \geq 1$, i.e.\ if
\[
f^m_n = \hbar^{n-1}\phi^m_n, \ m \geq 0,\ n \geq 1,
\]
for some $\phi^m_n : S^n(U') \to S^m(U'')[[\hbar]]$. 
Let therefore 
\[
f : S(U'')[[h]] \to  S(U''')[[h]]\ \mbox { and }
g : S(U')[[h]] \to  S(U'')[[h]]
\]
be such maps. We must prove that each $\flicek {{\, n}}m$-component
$(f\diamond g)^m_n$ of their
MV-composition~(\ref{Jaruska_mi_poslala_Mikulase}) with $n \geq 1$ is
divisible by
$\hbar^{n-1}$.  This clearly happens if the expression~(\ref{mikulas})
is, for $n =i_1 + \cdots + i_k$, divisible by $\hbar^{n-1}$.

Since each $f^{t_b}_{s_b}$ is divisible by $\hbar^{s_b-1}$ if $s_b
\geq 1$, $1 \leq b \leq l$, and each $g^{j_a}_{i_a}$ is divisible by
$\hbar^{i_a-1}$ if $i_a \geq 1$, $1 \leq a \leq k$ by
assumption,~(\ref{mikulas}) is divisible~by
\[
\hbar^{i_1 + \cdots + i_k + s_1 + \cdots +s_l - (\overline k+\overline
  l\,)} =
\hbar^{n + s_1 + \cdots +s_l- (\overline k+\overline l\,)},
\]
where $\overline k$ is the number of $a$'s for which $i_a \not=0$ 
(resp.~$\overline l$ the number of $b$'s for which $s_b
\not=0$). Thus~(\ref{mikulas}) is divisible by $\hbar^{n-1}$ if
\begin{equation}
\label{eq:za_tyden_do_Prahy}
n + s_1 + \cdots +s_l- (\overline k+\overline l\,) \geq n-1.
\end{equation}
Since $\overline k \leq k$ and $\overline l \leq
l$,~(\ref{eq:za_tyden_do_Prahy}) would follow from 
\[
n + s_1 + \cdots +s_l- (k+l) \geq n-1,
\]
which is the same as 
\[
s_1 + \cdots +s_l- (k+l) +1 \geq 0.
\]

Notice that~(\ref{mikulas}) is the sum over unshuffles $\kappa$ whose
associated graphs $\Gamma$ are connected. The number of vertices $V$
of such a graph
clearly equals $k+l$ while its number $E$ of edges is $s_1 + \cdots +s_l$.
We therefore need to prove that 
\begin{equation}
\label{JJR}
E-V+1 \geq 0.
\end{equation}
Since $\Gamma$ connected, $V-E = 1-b_1(\Gamma)$ by Euler's theorem or,
equivalently, $E-V +1 = b_1(\Gamma)$, where
$b_1(\Gamma)$ is the first Betti number of $\Gamma$. As
$b_1(\Gamma)\geq 0$,~(\ref{JJR})
immediately follows.
\end{proof}

\section{The Quantum Master Equation}
\label{sec:master-equation-1}

\begin{definition}
  Let $V$ be an \MV-algebra over $R$ with a maximal ideal $\frm$ as in
  Definition~\ref{Prvni_den_po_navratu_z_Minneapolis}.  The {\em
    quantum master equation (QME\/)\/} in $V$ is the equation
\begin{equation}
\label{eq:3}
\Delta e^S = 0,
\end{equation}
for a degree-$0$ element \hbox{$S \in V \hot\, \frm$}. We denote by
$\MC_V(R)$ the set of all solutions of the quantum master equation in $V$.
\end{definition}

In~(\ref{eq:3}), $e^S$ is the exponential in the graded commutative
associative algebra $V \hot\, R$. The existence of the exponential is
guaranteed by the completeness of $R$.

\begin{example}
[M\"unster-Sachs' formulation]
Let $A$ be an $\IBL_\infty$-algebra as in
Example~\ref{Treti_den_pekla}. M\"unster and Sachs consider $S \in
S(U)[[\hbar]]$ of the form
\begin{equation}
\label{Haifa}
S = \sum_{k \geq 2} \hbar^{k-1} c_k
\end{equation}
with some $c_k \in  \bigoplus_{1 \leq n \leq k} S^n(U)$.
\end{example}

\begin{remark}
  M\"unster and Sachs allowed $k=1$ in \eqref{Haifa}, which would be
  inaccurate in our setup, unless there were reasons to guarantee
  convergence of $e^S$ -- see the next example.
\end{remark}

\begin{example}[Formulation of \cite{sasha's}]
\label{previous}
Let $A$ be a $\BV_\infty$-algebra as in
Example~\ref{sec:mv-algebras-1}. Consider the functional field
\[
{\mathbb K} := \bfk(\!(\hbar)\!) = \bfk[[\hbar]][\hbar^{-1}].
\]
For an independent symbol $\lambda$, $\bfk[[\hbar]][[\lambda]] =
\bfk[[\hbar, \lambda]]$ is a complete local ring with residue field
$\bfk$ and ${\mathbb K}[[\lambda]]: = \bfk(\!(h)\!)[[\lambda]]$ is a
complete local ring with residue field~${\mathbb K}$. Extending
Remark~\ref{sec:mv-algebras-3}, we may consider $A$ as an
$\MV$-algebra over $\bfk[[\hbar, \lambda]]$ or over
$\bfk(\!(h)\!)[[\lambda]]$.

The work \cite{sasha's} treated the quantum master equation for $S$ of
the form $S = \tilde S/\hbar$, with some $\tilde S \in \lambda
A[[\hbar]][[\lambda]]$ or $\tilde S \in \lambda
A(\!(\hbar)\!)[[\lambda]]$. Here we work in our more general setup and
consider $\tilde S \in A(\!(\hbar)\!) \hot \, \frn$, where $\frn$ is
the maximal ideal in a complete local ring $\ring$ with residue field
$\bfk$.
\end{example}

\begin{theorem}[\cite{braun-lazarev:poisson}]
\label{sec:quant-mast-equat}
In the situation of Example~\ref{previous}, the equation $\Delta
e^{\tilde S/\hbar} = 0$ for $\tilde S \in A(\!(\hbar)\!) \hot \, \frn$
of degree zero is equivalent to
\[
\Delta \tilde S + \frac{1}{2!}  l_2(\tilde S,\tilde S) + \frac{1}{3!}
l_3(\tilde S,\tilde S,\tilde S) + \dots = 0,
\]
where $l_n$ for each $n \ge 2$ is the \emph{higher derived bracket\/}:
\begin{eqnarray*}
l_n (a_1, \dots, a_n) := 
 \frac{1}{\hbar^{n-1}}
\Phi_n^\Delta(a_1, \dots, a_n) 
 =  \sum_{k=1}^\infty \hbar^{k-1}
\Phi_n^{\Delta_{k+n-1}}(a_1, \dots, a_n),
\end{eqnarray*}
with $\Phi^?_n$ being defined in~$(\ref{psano_v_Haife})$.
\end{theorem}

\begin{proof}
  Let $\Ad$ denote the adjoint action $\Ad_g Y := g Y g^{-1}$ of the
  group $\GL(A)$ of invertible, degree-zero linear maps $g: A \to A$
  on the Lie algebra $\gl(A):= \Lin_\bfk (A,A)$ of all degree-zero
  linear maps $Y: A \to A$. Let $\ad$ be the adjoint action $\ad_X Y
  := [X,Y]$ of $\gl(A)$ on itself. Then, for $X \in \gl(A) \hot \,
  \frn$, we have
\begin{equation}
\label{ad}
\Ad_{e^X} = e^{\ad_X}.
\end{equation}

Given a degree-zero element $\tilde S \in A(\!(\hbar)\!) \hot \,
\frn$, apply Equation~\eqref{ad} to the operator $X = L_{(-\tilde
  S/\hbar)}$ of left multiplication by $-\tilde S/ \hbar$. We get
\[
\Delta \circ e^{\tilde S/\hbar} = e^{\tilde S/\hbar} \sum_{n=0}^\infty
\frac{\ad_{(-\tilde S/\hbar)}^n}{n!} \, \Delta = e^{\tilde S/\hbar}
\sum_{n=0}^\infty \frac{[[\ \dots [\Delta, L_{\tilde S/\hbar}], 
\hbox {$. \hskip .05em  . \hskip .05em.$}],
  L_{\tilde S /\hbar} ]}{n!},
\]
where in the last formula there are exactly $n$ iterated commutators
with $L_{\tilde S/\hbar}$. Applying both sides of this equation to $1
\in A$ and using Equation~\eqref{psano_v_Haife}, we obtain 
\[
\Delta (e^{\tilde S/\hbar}) = \frac{e^{\tilde S/\hbar}}{\hbar} \left(\Delta
\tilde{S} + \frac{1}{2!} l_2(\tilde S,\tilde S) + \frac{1}{3!}
l_3(\tilde S,\tilde S,\tilde S) + \cdots\right),
\]
whence the result.
\end{proof}

\begin{theorem}
\label{sec:master-equation-2}
There is a natural one-to-one correspondence between solutions $S$ of
the QME~\eqref{eq:3} in $V$ and \MV-mor\-phisms $s \in \ttMV_R(\bfk,V)$,
i.e.\ one has a natural isomorphism
\begin{equation}
\label{v_Koline_o_Velikonocich}
\MC_V(R) \cong \ttMV_R(\bfk,V),
\end{equation}
given by $s(1) = S$. In~\eqref{v_Koline_o_Velikonocich}, the MV
structure on\/ $\bfk$ is defined by $\delta(1) := 1 \ot 1$ and $\Delta
= 0$, as in Example~\ref{sec:mv-algebras-2}.
\end{theorem}

\begin{remark}[Geometric interpretation]
\label{representability}
  Using the geometric interpretation of \MV-mor\-phisms in
  Remark~\ref{prepsano_v_Tokyo}, we can interpret the proposition as
  stating that a solution of the QME is a family $X$ of \MV-manifolds
  over $B = \Spec R$, or rather a family of \MV-manifold structures on
  the trivial fiber bundle $X = \Spec V \htt B$. One can also view a
  solution as a~$B$-point $B \to X^*$ of the dual \MV-manifold.
\end{remark}

\begin{proof}[Proof of Theorem~\ref{sec:master-equation-2}]
  By definition, $\Lin_R^0(\bfk,V)$ consists of $R$-linear maps
\begin{equation}
\label{pes_steka}
s : \bfk\, \hot R \cong R \to V \hot R
\end{equation} 
such that $s(1) \in V \hot\, \frm$. Under the canonical isomorphism
\[
\Lin_R(\bfk,V) \cong V \hot R, \ s \ni \Lin_R(\bfk,V) \longmapsto
s(1) \in V \hot R,
\] 
maps with this property tautologically correspond to $V \hot\, \frm$.

A map $s$ in~(\ref{pes_steka}) is an \MV-mor\-phism if and only if
$\Delta\circ \exp(s)= 0$. Since $\Delta \circ \exp(s)$ is $R$-linear,
this happens if and only if $\Delta \exp(s)(1)= 0$. As $1 \in \bfk$ is
group-like,
\[
\Delta \exp(s)(1)= \Delta e^{s(1)} = \Delta e^S
\]
by Example~\ref{za_14_dni_Izrael}. 
\end{proof}

Let $V$ be an \MV-algebra over $\bfk[[\hbar]]$ and $\ring$ a complete
local ring with the maximal ideal~$\frn$ and residue field $\bfk$.
Extend the \MV\ structure on $V$ linearly, as in
Remark~\ref{sec:mv-algebras-3}, to an \MV\ structure on the same $V$
over $\bfk[[\hbar]] \hot\, \ring$ and $\bfk(\!(\hbar)\!) \hot\,
\ring$, regarded as complete local rings with residue fields $\bfk$
and ${\mathbb K} = \bfk(\!(\hbar)\!)$, respectively.

In the BV formalism of theoretical physics \cite{losev} and
applications to deformation theory \cite{iacono,kkp,terilla}, it is
important to consider the QME for elements $S$ of the form $S = \tilde
S/\hbar$ for $\tilde S \in V(\!(\hbar)\!) \hot\, \frn$ or $\tilde S
\in V[[\hbar]] \hot\, \frn$. We relate these types of solutions to
\MV-morphisms below.

\begin{corollary}
\begin{itemize}[leftmargin=1.8em,labelsep=0.4em,itemsep=.3em,rightmargin=1.8em]
\item[(i)] There is a natural one-to-one correspondence between
  solutions \hfill\break  
  $\tilde{S} \in V(\!(\hbar)\!) \hot\, \frn$ of the quantum master equation
\begin{equation}
\label{master-equation-1}
\Delta e^{\tilde{S}/\hbar} = 0
\end{equation}
and \MV-mor\-phisms $s \in \ttMV_{{\mathbb K} \hot\, \ring} ({\mathbb K},V)$,
i.e. one has a natural isomorphism
\begin{equation}
\label{v_Koline_o_Velikonocich-1}
\big\{\tilde{S} \in
V(\!(\hbar)\!) \hot\, \frn \, | \, \Delta e^{\tilde{S}/\hbar} = 0\big\} = 
\hbar \cdot \MC_{V(\!(\hbar)\!)} ({\mathbb K} \hot \ring) \xrightarrow{\sim} \ttMV_{{\mathbb K} \hot \, \ring}\big({\mathbb K},V(\!(\hbar)\!) \big),
\end{equation}
given by $s(1) = \tilde{S}/\hbar$.

\item[(ii)] There is a natural one-to-one correspondence between solutions
  $\tilde{S} \in V[[\hbar]] \hot\, \frn$ of the  quantum master equation
\begin{equation}
\label{master-equation-2}
\Delta e^{\tilde{S}/\hbar} = 0
\end{equation}
and a certain subset of \MV-mor\-phisms $s \in \ttMV_{{\mathbb K}
  \hot\, \ring} ({\mathbb K},V(\!(\hbar)\!) )$, namely one has a
natural isomorphism
\begin{equation}
\label{v_Koline_o_Velikonocich-2}
\big\{\tilde{S} \in
V[[\hbar]] \hot\, \frn \, | \, \Delta e^{\tilde{S}/\hbar} = 0\big\} \xrightarrow{\sim}
\ttMV_{{\mathbb K} \hot \, \ring}^s (\mathbb K,V(\!(\hbar)\!) ),
\end{equation}
given by $s(1) = \tilde{S}/\hbar$, where $\ttMV_{{\mathbb K} \hot \,
  \ring}^s (\mathbb K,V(\!(\hbar)\!) )$ denotes the set of
\MV-mor\-phisms $s$ whose value $s(1)$ at $1 \in {\mathbb K} \hot \,
\ring$ has at most a simple pole at~$\hbar = 0$.
\end{itemize}
\end{corollary}

\begin{remark}[Geometric interpretation]
  Now, as we have two local algebras, geometric interpretation becomes
  subtler. Let $D$ denote $\Spec \bfk[[\hbar]]$, often called the
  \emph{formal disk}, $\mathring{D} = \Spec \mathbb K$ be the deleted
  formal disk, and $C = \Spec T$. An \MV-algebra over $\bfk[[\hbar]]$
  has been interpreted as an \MV-manifold over $D$, which may be
  restricted to an \MV-manifold over $\mathring{D}$. Let us denote it $X_0
  \to \mathring{D}$. Then the
  isomorphism~\eqref{v_Koline_o_Velikonocich-1} interprets a solution
  of the modified QME~\eqref{master-equation-1} as a family of
  \MV-manifolds extending the family $X_0 \to \mathring{D}$ to the base
  $\mathring{D} \htt C$:
\[
\xymatrix{
X_0 \ar[r] \ar[d] & X \ar[d]\\
\mathring{D}  \ar[r] & \mathring{D} \htt C.
}
\]
In the dual interpretation, the
isomorphism~\eqref{v_Koline_o_Velikonocich-1} asserts that the dual
\MV-manifold $X_0^* \to \mathring{D}$ is an \MV-manifold representing
the functor solutions of QME~\eqref{master-equation-1} or the ``moduli
\MV-space'' of solutions of the QME. Indeed, a solution of
\eqref{master-equation-1} may be interpreted as an \MV-morphism
\[
\xymatrix{
\mathring{D} \htt C \ar[r] \ar[dr] & X_0^* \ar[d]\\
& \mathring{D}.
}
\]
\end{remark}

Theorem~\ref{sec:master-equation-2} treats solutions of the
QME~\eqref{eq:3} as morphisms in the category of $\MV$-algebras over
$R$. Thus, composition of \MV-morphisms can be used to transfer these
solutions over $\MV$-mor\-phisms. More specifically, let $f \in
\MV_R(V',V'')$ be a morphism and $S \in V' \hot\, \frm$ a solution of
the QME in $V'$.  Theorem~\ref{sec:master-equation-2} translates $S$
into a morphism $s \in \MV_R(\bfk,V')$. Then $f \diamond s$ will be a
morphism in $\MV_R(\bfk,V'')$ therefore, by
Theorem~\ref{sec:master-equation-2} again, $(f \diamond s)(1)$ will
solve the QME in $V''$. Since $1\in \bfk$ is group-like, one moreover
has
\[
(f \diamond s)(1) = \log\big((\exp(f) \circ \exp(s))(1) \big) =
\log\big(\exp(f) (e^{s(1)})\big) =\log\big(\exp(f) (e^S)\big),
\]
the `actual' logarithm of $\exp(f) (e^S) \in V'' \hot \frm$.

\begin{definition}
The element 
\[
f_!(S) : = \log\big(\exp(f) (e^S)\big) \in \MC_{V''}(R)
\]
is the {\em push-forward\/} of the solution $S \in \MC_{V'}(R)$.
\end{definition}

\begin{theorem}[Cf.\ \cite{sasha's} for the case $V' = S(U')$]
\label{BV_transfer}
Assume that $V'$ is a bialgebra, i.e.\ the standard compatibility
between the multiplication and the comultiplication is fulfilled. Suppose
moreover that $S \in P(V') \hot \frm$, where $P(V')$ is the subspace of
primitive elements in $V'$. Then
\[
f_!(S) = f(e^S).
\]
\end{theorem}

\begin{proof}
  Observe that $e^S$ is a group-like element in $V' \hot R$. Indeed,
\[
\delta'(e^S) = e^{\delta'(S)} = e^{(S \ot_R 1 + 1 \ot_R S)} = 
  e^{S \ot_R 1}e^{ 1 \ot_R S} = e^S \ot_R e^S,
\]
where we used the fact that $\delta'$ is an algebra morphism and that
$S \ot_R 1$ commutes with $ 1 \ot_R S$ in $(V' \hot R)\ot_R(V' \hot
R)$.  The rest follows from the observations in
Example~\ref{za_14_dni_Izrael}.
\end{proof}

\def\cprime{$'$}\def\cprime{$'$}

\end{document}